\documentclass[12pt,leqno]{amsart}
\usepackage{amsmath,stmaryrd}
\usepackage{amssymb}
\usepackage{amsthm}
\usepackage{epsf}
\usepackage{graphicx}
\usepackage{esint} 
\usepackage{dsfont}
\usepackage{tikz}
\usepackage[pagebackref]{hyperref} 
\hypersetup{pdfpagemode=FullScreen,  colorlinks=true} 
\usepackage{enumerate} 
\usepackage{xcolor,bbm,easybmat,bm}
\usepackage{accents}

\usepackage{amsfonts}
\usepackage{amsrefs}
\usepackage{verbatim}
\usepackage{booktabs}
\usepackage{color}
\usepackage{graphicx,float}
\usepackage{accents}
\newcommand{\dbtilde}[1]{\accentset{\approx}{#1}}

\newcommand{\tripprox}{\setbox0\hbox{$\approx$}%
\mbox{\makebox[0pt][l]{\raisebox{0.48\ht0}{$\approx$}}$\approx$}}
\newcommand{\trtilde}[1]{\accentset{\tripprox}{#1}}

\numberwithin{equation}{section}

\makeatletter
\newcommand\footnoteref[1]{\protected@xdef\@thefnmark{\ref{#1}}\@footnotemark}
\makeatother

\newcommand{\dr}{\partial}

\DeclareMathOperator{\divg}{div}

\DeclareMathOperator{\diam}{diam}

\DeclareMathOperator{\loc}{loc}

\newcommand{\1}{{\mathds 1}}

\newcommand{\R}{\mathbb R}
\newcommand{\N}{\mathbb N}
\newcommand{\cF}{\mathcal F}

\newcommand{\wt}{\widetilde}

\newcommand{\Base}{\mathcal{O}}

\renewcommand{\L}{L}

\newcommand{\C}{\mathcal C}

\newcommand{\Rn}{\mathbb R^n}

\newcommand{\abs}[1]{\left\vert#1\right\vert}
\newcommand{\br}[1]{\left(#1\right)}

\newcommand{\set}[1]{\left\{#1\right\}}
\renewcommand{\div}{\mathrm{div}}
\renewcommand{\d}{\, \mathrm{d}} %differential

\newcommand{\om}{\Omega}
\newcommand{\pom}{\partial\Omega}

\newcommand{\E}{\mathsf{E}} %Energy space
\newcommand{\Lloc}{\L_{\operatorname{loc}}} %local Lebesgue spaces
\newcommand{\HT}{H_t} %Hilbert transform
\newcommand{\dhalf}{D_t^{1/2}} %half time derivative
\newcommand{\Hdot}{\dot{H}\protect{\vphantom{H}}} %homogeneous Sobolev spaces
\newcommand{\mS}{{\mathcal S}} %Schwartz space/Single layer

\newcommand{\pd}{\partial}
\newcommand{\cl}[1]{\overline{#1}} %closure

\newcommand{\ree}{{\mathbb{R}^{n}}}

\newcommand{\lpformath}{\rm{(}}
\newcommand{\rpformath}{\rm{)}}

\newcommand{\Nformath}{\rm{N}}
\newcommand{\Dformath}{\rm{D}}

\newcommand{\Np}{\hyperlink{Np}{$\lpformath\Nformath_{p}\rpformath$}}
\newcommand{\Nq}{\hyperlink{Np}{$\lpformath\Nformath_{q}\rpformath$}}

\newcommand{\Dq}{\hyperlink{Dq}{$\lpformath\Dformath_{p'}\rpformath$}}

\newcommand{\dint}{\int\!\!\!\!\!\int}
\def\Yint#1{\mathchoice
	{\YYint\displaystyle\textstyle{#1}}%
	{\YYint\textstyle\scriptstyle{#1}}%
	{\YYint\scriptstyle\scriptscriptstyle{#1}}%
	{\YYint\scriptscriptstyle\scriptscriptstyle{#1}}%
	\!\dint}
\def\YYint#1#2#3{{\setbox0=\hbox{$#1{#2#3}{\iint}$}
		\vcenter{\hbox{$#2#3$}}\kern-.51\wd0}}
\def\longdash{\mkern-1.5mu{-}\mkern-7.5mu{-}} 
\def\fiint{\Yint\longdash}

\newcommand{\Z}{{\mathbb Z}}

\theoremstyle{plain}
\newtheorem{theorem}[equation]{Theorem}
\newtheorem{lemma}[equation]{Lemma}

\newtheorem{proposition}[equation]{Proposition}

\newtheorem{definition}[equation]{Definition}

\theoremstyle{definition}

\theoremstyle{remark}
\newtheorem{remark}[equation]{Remark}

\begin{document}

\title[Extrapolation of solvability of the parabolic $L^p$ Neumann problem ]{Extrapolation of solvability of the parabolic $L^p$ Neumann problem on bounded Lipschitz cylinders}

\author[Dindo\v{s}]{Martin Dindo\v{s}}
\address{School of Mathematics, 
The University of Edinburgh and Maxwell Institute of Mathematical Sciences, Edinburgh, UK}
\email{M.Dindos@ed.ac.uk}

\author[Liu]{YingYi Liu}
\address{School of Informatics, The University of Edinburgh, Edinburgh, UK}
\email{s2029366@ed.ac.uk}

%\thanks{The first author has been supported in part by EPSRC grant EP/Y033078/1.}


\maketitle

\begin{abstract} 
A recent result of the first author with Li and Pipher has established the extrapolation of solvability of the $L^p$ parabolic Neumann problem on unbounded graph domains of the form $\Omega=\{(x',x_n):\,x_n>\phi(x')\}\times\R$,
where $\phi:\R^{n-1}\to\R$ is a Lipschitz function. The result shows that under the assumptions that the $L^p$ parabolic Neumann problem for the equation $Lu=-\partial_t u+\divg(A\nabla u)=0$ in $\Omega$ and also 
the $L^{p'}$ parabolic Dirichlet problem for the adjoint equation $L^*u=\partial_t u+\divg(A^T\nabla u)=0$ in $\Omega$
are solvable, then also the $L^q$ parabolic Neumann problem for the equation $Lu=0$ in $\Omega$
is solvable for all $1<q<p$.

However, the mentioned paper does not answer the question whether the same claim is also true for domains of the form $\mathcal O\times\R$, where $\mathcal O$ is a bounded Lipschitz domain (in spatial variables) since this case does not follow from our argument for the unbounded case. Indeed, the bounded Lipschitz cylinder case requires a significantly different approach which we present in this article and establish an analogous result when $\mathcal O$ is a bounded Lipschitz domain.
\end{abstract}

\tableofcontents

%\ms\noindent{\bf Keywords:} 

%\ms\noindent

\section{Introduction}
The main objective of this paper is to establish extrapolation of solvability for the $L^p$ Neumann problem for the parabolic PDE
$$L u= -\dr_t u +\divg (A \nabla u)=0,\qquad\mbox{in } \Omega,$$
where the domain $\Omega$ is of the form $\mathcal O\times\R$, where $\mathcal O$ is a bounded Lipschitz domain (in spatial variables), i.e., a bounded (in space) Lipschitz cylinder that is infinite in the time variable. 

The coefficients of the elliptic matrix $A$ are assumed to be real, bounded, measurable and are allowed to be time varying. No other assumption of the coefficients is made. More precisely, we assume that $A= [a_{ij}(X, t)]$ is an $n\times n$ bounded matrix satisfying the following uniform ellipticity condition: there exist positive constants $\lambda$ and $\Lambda$ such that
\begin{equation}
	\label{E:elliptic}
	\lambda |\xi|^2 \leq \sum_{i,j} a_{ij}(X,t) \xi_i \xi_j,\qquad |A(X,t)| \leq \Lambda,
\end{equation}
for almost every $(X,t) \in \Omega$ and all $\xi \in \R^n$.\medskip

We prove that if for some $p\in(1,\infty)$ the $L^p$ Neumann problem for the PDE $Lu=0$ in $\Omega$
is solvable (and also the $L^{p'}$ Dirichlet problem for the adjoint PDE $L^*u=\dr_t u +\divg (A^T \nabla u)=0$ in $\Omega$ is solvable), then the $L^q$ Neumann problem for the PDE $Lu=0$ in $\Omega$ is solvable for all $1<q<p$. We note that $L^*$ is a time-backward operator, but as we impose boundary conditions on the lateral boundary $\partial\mathcal O\times\R$ this causes no issues, and the change of variables $\tau\mapsto -t$ gets us back to a standard time-forward parabolic operator.  
\medskip

The importance of this result is that it is \lq\lq clean" and does not impose any additional assumptions on the coefficients beyond the natural ones (ellipticity and boundedness). Results of this nature are known for the parabolic Dirichlet problems (either using known properties of $B_p$ weights or alternatively by interpolation against the $L^\infty$ end-point which is a consequence of the maximum principle) and also recently for the Regularity problem \cite{DiS} using an interpolation against a parabolic Hardy-Sobolev endpoint $p=1$. An analogous elliptic result for the Neumann problem is known and was first proven in \cite{KP93} under stronger assumptions (where the related Regularity problem must also be solvable). This was later relaxed in \cite{FL} to the solvability of the $L^{p'}$ Dirichlet problem for the adjoint PDE. Our claim only uses this weaker assumption. This assumption seems necessary, since even in the elliptic case there are no known results without this condition except on special domains, but even in these exceptional cases, the assumption is usually not \lq\lq violated"—rather, it is either structurally guaranteed (so it doesn't need to be posed as an additional hypothesis) or is bypassed entirely via different analytic machinery. 

Hence our result closes the gap in our understanding of parabolic boundary value problems and brings it to the same state of the art as in the elliptic case.
\medskip

The paper \cite{DLP3} has proven this result for unbounded graph domains of the form 
\begin{equation}\label{domain}
\Omega=\{(x',x_n):\,x_n>\phi(x')\}\times\R
\end{equation}
however, the case of bounded Lipschitz cylinders we consider here is significantly different and requires substantially new ideas. As we shall see, at one point we introduce an unbounded domain of the type \eqref{domain} and solve an auxiliary Neumann problem associated to it. We will have to overcome a significant obstacle there since we do not assume Neumann solvability on the domain \eqref{domain} but only on the original bounded
domain. We explain in detail why we introduce this intermediate domain in subsection \ref{ssUD}.
\medskip

We note that for specific operators (such as the heat equation on Lipschitz cylinders, or even more general domains), this result is known (cf.\ \cite{B}).  We also note that in the recent preprint \cite{DLP2}, the solvability of the $L^p$ Neumann problem 
in the range $1<p<p_0$ for some $p_0>1$ has been established on bounded Lipschitz cylinders, i.e., domains we consider here, but with additional assumptions, namely the coefficients $A$ satisfying a small Carleson measure condition and also smallness of the Lipschitz norm of the domain. Hence, the result presented here is of a different nature, as we make no smallness assumptions on the domain and no further assumptions on the coefficients.
 
 \begin{theorem}\label{MT} Let $\Omega=\mathcal O\times\R$, where 
$\mathcal O\subset\R^n$ is a bounded Lipschitz domain (as defined by Definition \ref{DefLipDomain}).
Consider the PDE 
$$L u= -\dr_t u +\divg (A \nabla u)=0   \quad\text{in } \Omega,$$ 
such that \eqref{E:elliptic} holds for a.e. $(X,t)\in \Omega$.
Assume that for some $p\in(1,\infty)$ the $L^p$ Neumann problem \Np$^L$  and the $L^{p'}$ Dirichlet problem  \Dq$^{L^*}$ for the adjoint PDE $L^*u=\dr_t u +\divg (A^T \nabla u)=0$ are both solvable (as defined in Section \ref{S2}).  
 
 Then for all $1<q<p$ the $L^q$ Neumann problem \Nq$^L$ is also solvable. Furthermore for $p=1$
 the problem is solvable for Neumann data in the atomic Hardy space $\hbar^1_{ato}(\partial\Omega)$.
\end{theorem}

\noindent{\it Remark.} Theorem \ref{MT} is a \emph{downward} extrapolation statement: from solvability at the fixed exponent $p$ it produces solvability for all smaller $1<q<p$, the mechanism being the atomic $p=1$ endpoint together with real interpolation. The complementary \emph{self-improvement} question --- whether \Np$^L$ also implies solvability of the $L^{p+\varepsilon}$ Neumann problem for some $\varepsilon>0$ --- is of a genuinely different (real-variable) nature and is not addressed by our argument. Such a self-improvement would follow from a reverse-H\"older/Gehring-type improvement of the non-tangential maximal function estimate (together with the corresponding self-improvement of the adjoint Dirichlet hypothesis \Dq$^{L^*}$), as is known for the elliptic Neumann problem \cite{KP93,FL} and for the parabolic Dirichlet and Regularity problems \cite{DiS}. We expect the analogous statement to hold in the present setting, but as our hypotheses are posed at a fixed $p$ and the downward conclusion already covers the entire range $1<q<p$, we do not pursue this here.

 \section{Definitions}\label{S2}

 \begin{definition}
A parabolic cube on $\R^n\times\R$  centered at $(X,t)$ with sidelength $r$ is defined as
$$    Q_r(X,t):=\{ (Y, s) \in \R^{n}\times\R : |x_i - y_i| < r \ \text{ for } 1 \leq i \leq n, \ | t - s |^{1/2} < r \},
$$
and we write $\ell(Q_r)=r$. Thus $Q_r(X,t)=Q_r(X)\times(t-r^2,t+r^2)$, where $Q_r(X):=\{Y\in\R^n:\,|x_i-y_i|<r\ \text{for } 1\le i\le n\}$ is the (open) spatial cube centered at $X$ with edges parallel to the coordinate axes and edge-length $2r$. Following the customary usage in the parabolic literature we refer to $Q_r(X,t)$ as a \emph{parabolic cube}, even though with respect to the parabolic metric it is a cylinder (the product of the spatial cube $Q_r(X)$ with a time interval).

A parabolic ball on $\R^n\times\R$  centered at $(X,t)$ with radius $r$ is the ball
\begin{equation}\label{eqdef.ball}
    B_r(X,t):=\{ (Y, s) \in \R^{n}\times\R : d_p((X,t),(Y,s))<r \},
\end{equation}
where $d_p(\cdot,\cdot)$ is the parabolic distance function
\[
d_p((X,t),(Y,s)) := \|(X-Y,t-s)\|,\]
where
 \[ \|(X,t)\|:=\br{\abs{X}^2+\abs{t}}^{1/2}.\]
%\sim \|(X-Y,t-s)\|.\]
For parabolic balls at the boundary we use notation $\Delta_r(X,t)=B_r(X,t)\cap \pom$. 

Finally, by $J_r(X,t)$ we denote the backward parabolic cylinder
$$J_r(X,t)=Q_r(X)\times (t-r^2,t)=Q_r(X,t)\cap\{s<t\},$$
the past half of the parabolic cube $Q_r(X,t)$.
\end{definition}

\begin{definition}
$\Z \subset \R^n$ is an $\ell$-cylinder of diameter $d$ if there
exists an orthogonal coordinate system $(x',x_n)$  with $x'\in\mathbb R^{n-1}$ and $x_n\in\mathbb R$ such that
\[
\Z = \{ (x',x_n)\; : \; |x'|\leq d, \; -(\ell+1) d \leq x_n \leq (\ell+1) d \}
\]
and for $s>0$,
\[
s\Z:=\{(x',x_n)\;:\; |x'|\le sd, -(\ell +1)s d \leq x_n \leq (\ell +1)s d \}.
\]
We do allow for the option that $d=\infty$ and in this case $s\Z=\Z=\R^n$.
\end{definition}

\begin{definition}\label{DefLipDomain}
$\mathcal O\subset \R^n$ is a Lipschitz domain with Lipschitz
`character' $(\ell,N,C_0)$ if there exists a positive scale $r_0\in (0,
\infty]$ and
at most $N$ $\ell$-cylinders $\{{\Z}_j\}_{j=1}^N$ of diameter $d$, with
$\frac{r_0}{C_0}\leq d \leq C_0 r_0$ such that 
\vglue2mm

\noindent (i) $8{\Z}_j \cap {\partial\mathcal O}$ is the graph of a Lipschitz
function $\phi_j$, $\|\nabla\phi_j \|_\infty \leq \ell \, ;
\phi_j(0)=0$,\vglue2mm

\noindent (ii) $\displaystyle {\partial\mathcal O}=\bigcup_j ({\Z}_j \cap {\partial\mathcal O}
)$,

\noindent (iii) $\displaystyle{\Z}_j \cap \mathcal O \supset \left\{
(x',x_n)\in\mathcal O \; : \; |x'|<d, \; \mathrm{dist}\left( (x',x_n),{\partial\mathcal O}
\right) \leq \frac{d}{2}\right\}$, where the right-hand side is written in the coordinate
system associated with $\Z_j$.

\noindent (iv) Each cylinder $\displaystyle{\Z}_j$ contains points from $\mathcal O^c={\mathbb R^n}\setminus\mathcal O$.

\noindent (v) If $r_0<\infty$ the domain $\mathcal O$ is a bounded set.
\vglue1mm
\end{definition}

\noindent{\it Remark.} If the scale $r_0$ is finite, then the domain $\mathcal O$ from the definition above is a bounded Lipschitz domain. The set $\mathcal O\times\R$ will be called a parabolic Lipschitz cylinder with Lipschitz base $\mathcal O$.

If we allow both $r_0,\, d$ to be infinite, then since $\displaystyle{\Z}=\mathbb R^n$,
we have that  $\mathcal O$ can be written in some coordinate system as
$$\mathcal O = \{(x',x_n): x_n > \phi(x')\}\quad\mbox{ where $ \phi:\mathbb R^{n-1} \rightarrow \mathbb R$ is a Lipschitz function.}$$
This is the case considered in \cite{DLP3} and therefore from now on we shall always assume that $0<d<\infty$.

\begin{definition}
For $a>0$ and $(q,\tau)\in\pom$, unless otherwise defined, we denote the non-tangential parabolic cones by
\begin{equation}\label{Gamma2.11}
    \Gamma_a(q,\tau):=\set{(X,t)\in\om: d_p((X,t),(q,\tau))<(1+a)\delta(X,t)},
\end{equation}
and $\delta(\cdot)$ is the parabolic distance to the boundary: 
\[
\delta(X, t) = \inf_{(q,\tau)\in\pom}
d_p((X, t),(q, \tau)).\]
\end{definition}

\begin{definition}
For $w\in L^\infty_{\loc}(\om)$, we define the non-tangential maximal function of $w$ as
\[
 N_a(w)(q,\tau):=\sup_{(X,t)\in\Gamma_a(q,\tau)}\abs{w(X,t)} \quad\text{for }(q,\tau)\in\pom.
\]
If $w\in L^2_{\loc}(\om)$, we define  the modified non-tangential maximal function
\begin{equation}\label{def.Nap}
    \wt N_{a}(w)(q,\tau):=\sup_{(X,t)\in\Gamma_a(q,\tau)}\br{\fiint_{B_{\delta(X,t)/2}(X,t)}\abs{w(Y,s)}^2dYds}^{1/2} \quad\text{for }(q,\tau)\in\pom.
\end{equation}
\end{definition}

\begin{remark}\label{rem.normalN}
When $\mathcal O$ is bounded, in Section \ref{S5} we shall use the following normalization of $\wt N$: the aperture satisfies $a\le \frac12$, and the averages in \eqref{def.Nap} are taken over the balls $B_{\eta\,\delta(X,t)}(X,t)$ in place of $B_{\delta(X,t)/2}(X,t)$, where $a$ and $\eta$ are chosen small depending only on $\diam(\mathcal O)$ so that for every $(q,\tau)\in\pom$ the value of $\wt N(w)(q,\tau)$ depends only on $w$ restricted to $\mathcal O\times(\tau-1,\tau+1)$. This is possible: if $(Y,s)\in\Gamma_a(q,\tau)$ then $\delta(Y,s)\le |Y-q|$ (as $(q,\tau)\in\pom$), so
$$|s-\tau|<(1+a)^2\delta(Y,s)^2-|Y-q|^2\le \big((1+a)^2-1\big)|Y-q|^2\le 3a\diam(\mathcal O)^2,$$
while the averaging balls add at most $\eta^2\diam(\mathcal O)^2$ to the time reach; it thus suffices to take $3a\diam(\mathcal O)^2\le\frac12$ and $\eta^2\diam(\mathcal O)^2\le\frac12$. Since the $L^p(\pom)$ norms of $\wt N$ corresponding to different apertures and different fractions in the averaging balls are comparable
(by a standard level-set argument; see e.g.\ \cite[Chapter~1]{Ke94}), with constants depending only on the Lipschitz character of the domain, all the estimates in Section \ref{S5} are unaffected by this normalization.
\end{remark}

\subsection{Energy solutions}\label{RwEs}

We recall the papers \cite{AEN,Din23} where these concepts are discussed in greater detail. With $\Omega=\mathcal O\times\mathbb R$, where $\mathcal O$ is either a bounded or an unbounded Lipschitz domain,
 we say that $v:\Omega\to\R$ belongs to the \emph{energy class} $\dot \E(\mathcal O\times\mathbb R)$ if
\begin{align*}
 \|v\|_{\dot \E} := \bigg(\|\nabla v\|_{\L^2(\mathcal O\times\mathbb R)}^2 + \|\HT \dhalf v\|_{\L^2(\mathcal O\times\mathbb R)}^2 \bigg)^{1/2} < \infty.
\end{align*}
Consequently, these are called \emph{energy solutions}. When considered modulo constants, $\dot \E $ is a Hilbert space and it is in fact the closure of $\C_0^\infty\!\big(\,\cl{\mathcal O\times\mathbb R}\,\big)$ for the homogeneous norm $\|\cdot\|_{\dot \E}$
and it coincides with the space $\dot{L}^2_{1,1/2}(\Omega)$.

As shown in \cite{AEN} (with a small generalization), functions from $\dot \E $ have well defined Dirichlet traces with values in
 the \emph{homogeneous parabolic Sobolev space} $\Hdot^{1/4}_{\pd_{t} - \Delta_x}(\partial\mathcal O\times\mathbb R)$. 
 
 Here, $\Hdot^{s}_{\pm \pd_{t} - \Delta_x}(\R^n)$ is defined as the closure of Schwartz functions $v \in \mS(\ree)$ with Fourier support away from the origin in the norm $\|\cF^{-1}((|\xi|^2 \pm i \tau)^s \cF v)\|_2$. This yields a space of tempered  distributions modulo constants in $\Lloc^2(\ree)$ if $0 < s \leq 1/2$. 
 Conversely, any $g \in \Hdot^{1/4}_{\pd_{t} - \Delta_x}(\R^n)$ can be extended to a function $v \in  \dot \E(\R^n_+\times\R)$ with trace $v\big|_{\partial\mathbb R^{n}_+\times\R} = g$.  
 
 Analogously, the space $\Hdot^{1/4}_{\pd_{t} - \Delta_x}(\partial\mathcal O\times\mathbb R)$ for a Lipschitz
 $\phi:\R^{n-1}\to\R$ with
$$ \partial\mathcal O=\{(x',x_n):\,x_n=\phi(x')\},$$
 can be defined via the projection $(x',x_n,t)\mapsto (x',t)$ from $\partial\mathcal O\times\R\to\R^{n-1}\times\R$ and once again $\Hdot^{1/4}_{\pd_{t} - \Delta_x}(\partial\mathcal O\times\mathbb R)$ is exactly the space of traces of functions from $\dot \E(\mathcal O\times\mathbb R)$.
 Finally,  via partition of unity we can then define 
 $\Hdot^{1/4}_{\pd_{t} - \Delta_x}(\partial\mathcal O\times\mathbb R)$ for $\mathcal O$ a bounded Lipschitz domain.\medskip
  
Hence, by the energy solution to $-\partial_tu + \divg(A\nabla u)=0$ with Dirichlet boundary datum $u\big|_{\partial\mathcal O\times\R} = f \in \Hdot^{1/4}_{\pd_{t} - \Delta_x}(\partial\mathcal O\times\mathbb R)$ (understood in the trace sense) we mean $u \in \dot\E(\mathcal O\times\mathbb R)$ such that
\begin{align*}
a(u,v):= \iint_{\mathcal O\times\R} \left[A \nabla u \cdot{\nabla v} + \HT \dhalf u \cdot {\dhalf v}\right] \d X \d t = 0,
\end{align*}
holds for all $v \in \dot \E_0$, the subspace of $\dot \E$ with zero boundary trace.

Moving onto the Neumann problem, given any $u \in \dot \E(\Omega)$, the co-normal derivative $\partial^A_\nu u\Big|_{\partial\Omega}:=\langle A\nabla u,\nu\rangle \Big|_{\partial\Omega}=g$ is defined via the formula
\begin{align}\label{eq.NeumannBdy}
 \iint_{\mathcal O\times\R} \left[A \nabla u \cdot{\nabla v} + \HT \dhalf u \cdot {\dhalf v}\right] \d X \d t -\int_{\partial\mathcal O
 \times\R}gv \d x\d t= 0,
\end{align}
for all $v \in \dot \E$. Here, since the traces of $v$ belong to $\Hdot^{1/4}_{\pd_{t} - \Delta_x}(\partial\Omega)$ and all elements of the space $\Hdot^{1/4}_{\pd_{t} - \Delta_x}$ are realized by some $v \in \dot \E$,
 the Neumann boundary data must by duality  naturally belong to the space $\Hdot^{-1/4}_{\pd_{t} - \Delta_x}(\partial\Omega)$.

By \cite{AEN}, the key to solving these problems is the introduction of the modified sesquilinear form (introduced earlier in \cite{Ny2016}):
\begin{equation}\label{eq-sesq}
 a_\delta(u,v) := \iint_{\mathcal O\times\R} \left[A \nabla u \cdot {\nabla (1+\delta \HT) v} + \HT \dhalf u \cdot {\dhalf (1+\delta \HT) v}\right] \d X \d t,
\end{equation}
where $\delta$ is a  real number yet to be chosen. The Hilbert transform $\HT$ is a skew-symmetric isometric operator with inverse $-\HT$ on both $\dot \E$ and $\Hdot^{1/4}_{\pd_{t} - \Delta_x}$.  Hence, $1+\delta \HT$ is invertible on these spaces for any $\delta \in \R$. Hence for a fixed $\delta>0$  small enough, $a_\delta$ is coercive on $\dot \E$ since
\begin{equation}\label{eq:coer}
 a_\delta(u,u) \ge (\lambda-\Lambda\delta )\|\nabla u\|_2^2 + \delta \|\HT \dhalf u \|_2^2.
\end{equation} 
In particular $\delta=\lambda/(\Lambda+1)$ would work.
To solve the Dirichlet problem we take an extension $w \in \dot \E$ of the data $f$ and apply the Lax-Milgram lemma to $a_\delta$ on $\dot \E_0$ to obtain some $u \in \dot \E_0$ such that 
\begin{align*}
 a_\delta(u,v) = - a_{\delta}(w,v) \qquad (v \in \dot \E_0).          \end{align*}
Hence, $u + w$ is an energy solution with data $f$. Should there exist another solution $v$, then $a_\delta(u+w-v,u+w-v) = 0$ and hence by coercivity $\|u +w - v \|_{\dot \E} = 0$. Thus the two solutions only differ by a constant. It means that the Dirichlet problem associated with our parabolic PDE is \emph{well-posed} in the energy class. Similar arguments allow us to solve the Neumann problem by considering for the datum $g \in \Hdot^{-1/4}_{\pd_{t} - \Delta_x}(\partial\Omega)$ the solution $u
\in \dot\E$ such that 
\begin{align*}
 a_\delta(u,v) = \langle g, \mbox{Tr } (1+\delta \HT)v\rangle\qquad (v \in \dot \E).          \end{align*}

\begin{definition}[\Dq]\label{def.Dq}\hypertarget{Dq}{}
    Let $p\in(1,\infty)$. We say that the $L^{p'}$ Dirichlet problem is solvable for $L^*$, denoted by \Dq or \Dq$^{L^*}$, if there exists a constant $C>0$ such that for all $g \in \Hdot^{1/4}_{\pd_{t} - \Delta_x}(\pom)\cap L^{p'}(\pom)$, the energy solution $u\in \dot\E(\om)$ to $L^* u=\partial_tu+\divg(A^T\nabla u) = 0$ in $\om$ with trace $u|_{\pom}=g$ satisfies the estimate
    \[
    \|N(u)\|_{L^{p'}(\pom)}\le C\|g\|_{L^{p'}(\pom)}.
    \]
\end{definition}
\begin{definition}[\Np]\label{def.Np}\hypertarget{Np}{}
     Let $p\in(1,\infty)$. We say that the $L^{p}$ Neumann problem is solvable for $L$, denoted by \Np or \Np$^{L}$, if there exists a constant $C>0$ such that for all $g \in \Hdot^{-1/4}_{\pd_{t} - \Delta_x}(\partial\Omega)\cap L^p(\pom)$, the energy solution $u\in \dot\E(\om)$ to $Lu = 0$ in $\om$ with $\dr_\nu^Au|_{\pom}=g$ (defined as in \eqref{eq.NeumannBdy}) satisfies the estimate
    \[
    \|\wt N(\nabla u)\|_{L^p(\pom)}\le C\|g\|_{L^p(\pom)}.
    \]
\end{definition}
 
\section{Basic estimates and known results}
We assume $\om=\mathcal O\times\R$ where $\mathcal O$ is a Lipschitz domain. We recall four results shown in \cite{DLP3}, the first being a version of Caccioppoli's inequality, the second a result on boundedness of solutions near the boundary, the third the main theorem of the paper about bounds for $\wt N$ for the Neumann problem, and the fourth an interesting version of the Poincar\'e inequality for parabolic solutions.

\begin{lemma}[Caccioppoli's inequality]\label{L:Caccio}
 Let $J=J_r(X,t)$ be a backward parabolic cylinder that either satisfies $J_{4r}=J_{4r}(X,t)\subset\om$ or $(X,t)\in\pom$ (i.e. it is centered at $\pom$). Let $u$ be a weak solution of $Lu=0$ in $J_{4r}(X,t)\cap\om$, and assume that $u$ has zero Neumann data on $J_{4r}\cap\pom$ if $(X,t)\in\pom$.
	Then there exists a constant $C=C(\lambda,\Lambda, n)$ such that for any constant $c$,
	\[
	    \iint_{J_{r}(X, t)\cap\om} |\nabla u|^{2} dYds 
			\leq \frac{C}{r^2} \iint_{J_{2r}(X, t)\cap\om} |u(Y,s)-c|^{2} dYds.
	\]
\end{lemma}

\begin{lemma}\label{lem.MoserNeu}
Let $J_r$ and $u$ be as in Lemma~\ref{L:Caccio}.
    %Let $(x,t)\in\pom$. Let $u$ be a weak solution to \eqref{E:pde} in $Q_{4r}(x,t)\cap\om$ with zero Neumann data on $Q_{4r}(x,t)\cap\pom$. 
    There exists constant $C=C(\lambda,\Lambda,n)$ such that 
    \begin{equation}\label{eq.localbdd}
          \sup_{J_{r/2}\cap\om}|u| \le \frac{C}{r^{n+2}}\iint_{J_{2r}\cap\om}|u(Y,s)|dYds.
    \end{equation}
  \end{lemma}

\begin{theorem}\label{thm.NtoLoc}
Let $\om=\mathcal O\times\R$ where $\mathcal O$ is either a bounded or unbounded Lipschitz domain in $\Rn$. Let $L=-\dr_t +\divg (A\nabla \cdot)$, and let $L^*=\dr_t+\divg(A^T\nabla\cdot)$ be the adjoint operator of $L$.
    Let $p\in (1,\infty)$. Suppose that \Np$^L$ and \Dq$^{L^*}$ are solvable in $\om$ (see Definitions~\ref{def.Dq}, \ref{def.Np}), then for any backward parabolic cylinder $J_r(X,t)$ centered at some $(X,t)\in\pom$, and any local weak solution $u$ to $Lu=0$ in $J_{2r}(X,t)\cap\om$  with zero Neumann data on $J_{2r}\cap\pom$, we have
\begin{equation*} 
\|\wt N(\nabla u\1_{J_r})\|_{L^p(\partial \Omega)} \leq C r^{(n+1)/p} \fiint_{J_{2r}\cap \om} |\nabla u|\, dYds.
\end{equation*}
\end{theorem}

Here, we have modified the statement of this theorem slightly (by taking the average of $|\nabla u|$ instead of $|\nabla u|^2$ but the claims are equivalent). \medskip

Finally, we also have:

\begin{lemma}[Poincar\'e inequality for solutions]\label{lem.Pnc_sol} 
Let $Q_0=Q_0'\times (T_0,T_1)$ be a cube satisfying $Q_0\cap\om\neq\emptyset$. 
Let $u$ be a weak solution of $Lu=-\div  F$ in $Q_0\cap\om$, and assume in addition that $u$ has zero Neumann data on $Q_0\cap\pom$ if $Q_0\cap\pom\neq\emptyset$.
Let $(X,t)\in Q_0\cap\overline{\om}$ and let $J_r:=J_r(X,t)$ be a backward in time parabolic cylinder that satisfies $J_{4r}\subset Q_0$.  

Then there is some constant $C$ such that for $c_u=\fiint_{J_r\cap\om}u\,dXdt$ we have that
\[
\iint_{J_r\cap\om}|u-c_u|dXdt\le C r\iint_{J_{2r}\cap\om}\br{|\nabla u|+|  F|}dXdt.
\]
\end{lemma}

\section{Proof of Theorem~\ref{MT}}\label{S5}

\subsection{Strategy of the proof}
We assume that, for some $p\in(1,\infty)$, both \Np$^L$ and \Dq$^{L^*}$ are solvable. 
The domain is again of the form $\Omega=\mathcal O\times\R$, where $\mathcal O$ is a bounded Lipschitz domain.  We implement the following strategy to prove solvability of \Nq$^L$ for $1<q<p$:

We shall establish the end-point $p=1$ Hardy space bound of the form
\begin{equation}\label{AE}
\|\wt N(\nabla u)\|_{L^1(\partial\Omega)}\le C,
\end{equation}
for all $u$ that solve $Lu=0$ in $\Omega$ with Neumann boundary data $g$, where $g$ is any
$L^\infty$ atom, i.e.,  
$$\mbox{supp }g\subset Q_r(X,t)\cap \partial\Omega\quad\mbox{for some $(X,t)\in \partial\Omega$ and $r>0$, }
\|g\|_{L^\infty}\le |Q_r(X,t)\cap \partial\Omega|^{-1},\,\int_{\partial\Omega}g=0.$$
This bound implies solvability of the Neumann boundary value problem for $p=1$ for the operator $L$ with Hardy space data. Then by interpolation we get for all $1<q<p$ solvability of the \Nq$^{L}$, which proves the statement of Theorem \ref{MT} for the operator $L$.  Here we interpolate the sublinear functional
$$T:g\mapsto \wt N(\nabla u),$$
using the real interpolation method. Thanks to \eqref{AE} we know that $T$ is bounded from $\hbar^1_{ato}(\partial\Omega)\to L^1(\partial\Omega)$, where $\hbar^1_{ato}(\partial\Omega)$ is the atomic
Hardy space consisting of functions of the form $\sum_i\lambda_ig_i$ such that $\sum_i|\lambda_i|<\infty$ and each $g_i$ is an $L^\infty$ atom as defined above. $T$ is also bounded from $L^p(\partial\Omega)\to L^p(\partial\Omega)$ since we assume that \Np$^L$  is solvable. It then follows that $T$ is bounded as an operator on $L^q(\partial\Omega)$ for all $1<q<p$, thus implying solvability of \Nq$^L$.  \medskip

Recalling the definition of the Lipschitz domain
$\mathcal O\subset\R^n$ from Definition \ref{DefLipDomain}, a simple compactness argument implies that there exists $s_0<1$ such that the $\ell$-cylinders $s_0\Z_j$ still cover $\partial\mathcal O$. Thus by making the scale $r_0$ and diameter $d$ smaller if necessary (at the expense of increased $N$) we may without loss of generality assume that the union of $(1/2)\Z_j$ covers $\partial\mathcal O$.
 
 Then by rescaling the PDE, it may also be assumed that the $\ell$-cylinders $\Z_j$ in the definition above have diameter $d=1$. Let $N,\, C_0$ be as above and hence 
 there are $N$ such $\ell$-cylinders $(1/2)\Z_j$ needed to cover the boundary $\partial\mathcal O$.

Then by rotation and translation, we may assume that the atom $g$ is supported in a parabolic boundary cube
$Q_r(0,0)\cap \partial\Omega$, has zero average and that $d=1$ is the diameter of the $\ell$-cylinders
which define the set  $\mathcal O$ (c.f. Definition \ref{DefLipDomain}). Throughout this section $\wt N$ denotes the modified non-tangential maximal function normalized as in Remark \ref{rem.normalN}, so that its value at $(q,\tau)\in\pom$ only depends on the function on $\mathcal O\times(\tau-1,\tau+1)$.

\subsection{Large atoms}

We shall consider two cases. Firstly, the easier case when $r\gtrsim d=1$. In this case, clearly we have that
$$Q_r(0,0)\cap \partial\Omega\subset \partial\mathcal O\times (-r^2,r^2)\quad\mbox{and}\quad 
|Q_r(0,0)\cap \partial\Omega|\sim |\partial\mathcal O\times (-r^2,r^2)|,$$
and hence we might think about the atom $g$ as supported in 
$\partial\mathcal O\times (-r^2,r^2)$ with 
$$\|g\|_{L^\infty}\lesssim |\partial\mathcal O\times (-r^2,r^2)|^{-1},\qquad \int_{\pom}g=0.$$
Here $r\ge 1$.\medskip

We start with estimating $\wt N(\nabla u)$ in a neighbourhood of $\partial\mathcal O\times (-r^2,r^2)$.
By H\"older's inequality  and the $L^p$ Neumann solvability \Np$^{L}$ we have that
\begin{multline}\label{eq5.5}
\|\wt N(\nabla u) \|_{L^1(\partial\mathcal O\times (-4r^2,4r^2))}\le |\partial\mathcal O\times (-4r^2,4r^2)|^{1/p'}\left(\int_{\partial\mathcal O\times (-4r^2,4r^2)} \wt N(\nabla u)^p d\sigma\right)^{1/p}\\
\lesssim  |\partial\mathcal O\times (-r^2,r^2)|^{1/p'} \left(\int_{\partial\Omega}|g|^p\right)^{1/p}\le  
|\partial\mathcal O\times (-r^2,r^2)||\partial\mathcal O\times (-r^2,r^2)|^{-1}\le C. 
\end{multline}
Observe also that since $\partial_\nu^{A}u=0$ at the boundary for $t<-r^2$, we clearly must have $u\equiv c$ and hence $\nabla u\equiv 0$ for all $t<-r^2$. From this we clearly have that
$\wt N(\nabla u) (X,t)=0$ for $(X,t)\in\pom$ with $t<-4r^2$: by the normalization of $\wt N$ from Remark \ref{rem.normalN} its value at $(X,t)$ only depends on $\nabla u$ on $\mathcal O\times(t-1,t+1)$, and $-4r^2+1\le -r^2$ since $r\ge 1$.

It remains to estimate the $L^1$ norm of $\wt N(\nabla u)$ when $t>4r^2$. Here again we have
$\partial_\nu^{A}u=0$ (for all $t>r^2$) as we are outside the support of $g$.  We claim the following lemma holds (see also {\cite[Lemma 5.10]{DiS}} for a version of this lemma that applies to the Dirichlet problem):

\begin{lemma}\label{lemma:Exponential Decay} Let $p>1$. Assume that for the operator $L=-\partial_t+\divg(A\nabla\cdot)$, such that $A$ satisfies \eqref{E:elliptic} we have solvability of both 
\Np$^L$ and \Dq$^{L^*}$. 

 Let $v$ be a solution to $\L v=0$ in $\Omega$ and
    suppose that $\partial^A_\nu v=0$  on \( \partial\Base \times (0,\infty) \).
    There exists an \( \alpha=\alpha(n,\lambda,\Lambda,\mathcal O)>0 \) such that for all integers $k>0$ we have that
    \begin{align*}
        \| {\wt N}(\nabla v) \|_{L^1(\Delta_k)}
        \lesssim e^{-\alpha k} \| {\wt N}(\nabla v) \|_{L^1(\Delta_{0})}.     
    \end{align*}
Here $\Delta_i=\partial\mathcal O\times (i,i+1]$ and $\wt N=\wt N_{a}$ is the maximal function normalized as in Remark \ref{rem.normalN}, so that
the value of $\wt N_{a}(\nabla v)(X,t)$ only depends on the values of $\nabla v$ on the subset
$\mathcal O\times (t-1,t+1)$.
\end{lemma}

We claim the lemma above follows from Theorem \ref{thm.NtoLoc} and the following proposition.

\begin{proposition}\label{prop:Smooth Exponential Decay}
    Let \( u \) be such that \( {L}u = 0 \) in \( \Base \times [0,\infty) \) and
    moreover suppose that \( \partial_\nu^A u=0 \) on \( \partial \Base \times [0,\infty) \).
    Then there exists a \( \beta >0 \) such that for all \( t>0 \) and $c=\fint_{\mathcal O} u(\cdot,0)$
    \begin{align}
        \| u(\cdot,t)-c \|_{L^2(\mathcal{O})} 
        \lesssim e^{-\beta t} \| u(\cdot,0) -c\|_{L^2(\mathcal{O})}
        \lesssim e^{-\beta t} \| \nabla u(\cdot,0)\|_{L^2(\mathcal{O})}
         \label{eq:Exponential Decay}.
    \end{align}
\end{proposition}

\begin{proof}
    Set \( \eta(t) = \| u(\cdot,t)-c \|_{L^2(\Base)} \), where $c$ is the average of $u$ at $t=0$. It follows that
    $c(t)=\fint_{\mathcal O} u(\cdot,t)$ is constant in $t$ and hence equals to $c$ for all $t>0$, since
    $$\partial_t c(t)=\fint_{\mathcal O\times\{t\}}\partial_t u=\fint_{\mathcal O\times\{t\}}\divg(A\nabla u)=\frac{1}{|\mathcal O|}\int_{\partial\mathcal O\times\{t\}}\partial_\nu^A u=0.$$    
    And for \( t > 0 \)
    \begin{align*}
        2\eta(t)\eta'(t) &= \partial_t (\eta(t)^2) = 2\int_{\mathcal O\times\{t\}} (u-c)\partial_t u
        = - 2 \int_{\mathcal O\times\{t\}} A \nabla u \cdot \nabla u
        \leq -2\lambda \int_{\mathcal O\times\{t\}} |\nabla u|^2.
    \end{align*}
    Since the average of $u$ on $\mathcal O\times\{t\}$ equals $c$, applying the Poincar\'e inequality on $\mathcal O\times\{t\}$ 
        we have that 
    \begin{align*}
        -\int_{\mathcal O\times\{t\}} |\nabla u|^2 
        \lesssim - \int_{\mathcal O\times\{t\}} |u-c|^2
        = -\eta(t)^2.
    \end{align*}
    Hence \( \eta'(t) \lesssim -\eta(t) \)
        and we may apply Gronwall's inequality to deduce \eqref{eq:Exponential Decay}. The last inequality then holds thanks to the Poincar\'e inequality. 
\end{proof}

Observe that this calculation holds on very general domains $\mathcal O$ as it only requires the Poincar\'e inequality and the constant in the final estimate only depends on the ellipticity constant of the matrix $A$ and the Poincar\'e constant of the domain $\mathcal O$.\vglue1mm

\noindent {\it Proof of Lemma \ref{lemma:Exponential Decay}:} Assume that $\|{\wt N}(\nabla v) \|_{L^1(\Delta_{0})}<\infty$ and that $\partial_\nu^Av\big|_{\partial\Omega}\equiv0$ for all times $\ge 0$. Denote by \( \Omega_i := \Base \times (i,i+1] \) and \( \Omega^s_i := \Base \times (i+\frac14,i+\frac34] \).
Clearly, then  using the reverse H\"older inequality for the gradient (a consequence of Caccioppoli's inequality):
$$\left(\int_{\Omega^s_{0}}|\nabla v|^2\right)^{1/2}\lesssim \int_{\Omega_{0}}|\nabla v|\lesssim \| {\wt N}(\nabla v) \|_{L^1(\Delta_{0})}.$$
Hence using the condition $\partial_\nu^A v=0$ for $t>0$ and $L^2$ integrability of the gradient (as a consequence of the Caccioppoli's inequality) we get that for some $\tau\in (\frac14,\frac34)$
$$\left(\int_{\mathcal O\times\{\tau\}}|\nabla v|^2\right)^{1/2}\lesssim \|{\wt N}(\nabla v) \|_{L^1(\Delta_{0})}<\infty.$$
Hence, by the Poincar\'e inequality on $\mathcal O\times\{\tau\}$ (for $c$ as in Proposition \ref{prop:Smooth Exponential Decay})
$$\left(\int_{\mathcal O\times\{\tau\}}|v-c|^2\right)^{1/2}\lesssim \| {\wt N}(\nabla v) \|_{L^1(\Delta_{0})}<\infty.$$

Fix now some $k>3$. We propagate the $L^2$ initial data given on $\mathcal O\times\{\tau\}$ using  Proposition \ref{prop:Smooth Exponential Decay} to the interval $t\in (k-3,k+4)$. It follows that for all such $t$ we have 

$$\| v(\cdot,t)-c \|_{L^2(\mathcal{O})} \lesssim e^{-\beta k} \| v(\cdot,\tau)-c \|_{L^2(\mathcal{O})}
 \lesssim e^{-\beta k} \|{\wt N}(\nabla v) \|_{L^1(\Delta_{0})}<\infty.$$
By integrating over the interval $(k-3,k+4)$ and then applying Lemma \ref{lem.MoserNeu} for solutions 
\begin{equation}\label{zmez}
\sup_{\Omega_{k-2}\cup\Omega_{k-1}\cup\Omega_{k}\cup\Omega_{k+1}\cup\Omega_{k+2}}|v-c|\lesssim e^{-\beta k} \| {\wt N}(\nabla v) \|_{L^1(\Delta_{0})}<\infty.
\end{equation}
Since by boundary Caccioppoli's inequality $\sup_{\Omega_{k-2}\cup\Omega_{k-1}\cup\Omega_{k}\cup\Omega_{k+1}\cup\Omega_{k+2}}|v-c|$
controls the $L^2$ norm of $\nabla v$ over the smaller set $\mathcal O\times (k-1, k+2)$, we may now bound $\wt N(\nabla v)$ on $\Delta_k$. Recall that by the normalization of Remark \ref{rem.normalN} the value of $\wt N(\nabla v)$ on $\Delta_k$ only depends on $\nabla v$ on $\mathcal O\times(k-1,k+2)$. We split each cone into its boundary part, where $\delta(Y,s)<\frac18$, and its interior part, where $\delta(Y,s)\ge\frac18$. Since the aperture satisfies $a\le\frac12$, the boundary part of a cone with vertex $(q,\tau)\in\Delta_k$ consists of points $(Y,s)$ with $|s-\tau|<(1+a)^2\delta(Y,s)^2<\frac1{16}$, and is therefore contained in the union of boundedly many backwards parabolic cylinders $J_{1/4}(X_i,t_i)$ centered at points $(X_i,t_i)\in\partial\mathcal O\times(k-\frac18,k+\frac98]$. The enlarged cylinders $J_{1/2}(X_i,t_i)$ span times in $(k-\frac38,k+\frac98]\subset (k-1,k+2)\cap(0,\infty)$, so the Neumann data of $v$ vanish there and Theorem \ref{thm.NtoLoc} applies to each $J_{1/4}(X_i,t_i)$, yielding (in combination with the $L^2$ bound for $\nabla v$ above)
$$\|\wt N(\nabla v\1_{J_{1/4}(X_i,t_i)})\|_{L^p(\partial\Omega)}\lesssim \fiint_{J_{1/2}(X_i,t_i)\cap\om}|\nabla v|\,dYds\lesssim e^{-\beta k}\|{\wt N}(\nabla v)\|_{L^1(\Delta_0)}.$$
Summing over the boundedly many $i$ bounds the contribution of the boundary parts of the cones. For the interior parts, the interior Caccioppoli inequality combined with \eqref{zmez} bounds $\big(\fiint_{B_{\eta\delta(Y,s)}(Y,s)}|\nabla v|^2\big)^{1/2}$ by $Ce^{-\beta k}\|{\wt N}(\nabla v)\|_{L^1(\Delta_0)}$ at every point of the dependence set $\mathcal O\times(k-1,k+2)$ with $\delta(Y,s)\ge\frac18$, so the interior part of $\wt N(\nabla v)$ satisfies the same bound pointwise on $\Delta_k$. Since $|\Delta_k|=O(1)$, combining the two contributions gives for all $k>3$:
$$
\left(\int_{\Delta_k} {\wt N}(\nabla v)^p\right)^{1/p} \lesssim e^{-\beta k}\| {\wt N}(\nabla v )\|_{L^1(\Delta_{0})},
$$
from which our claim follows as $\|{\wt N}(\nabla v )\|_{L^1(\Delta_{k})}\lesssim \|{\wt N}(\nabla v )\|_{L^p(\Delta_{k})}$. Notice that we have excluded $k=1,2,3$, but for these values of $k$ we do not have to establish any decay: the bound
$\|\wt N(\nabla v)\|_{L^1(\Delta_k)}\lesssim \|\wt N(\nabla v)\|_{L^1(\Delta_0)}$ (without the exponential gain) holds for $k=1,2,3$ directly. Indeed, boundary H\"older continuity of solutions and the maximum principle bound
$\sup_{\Omega_{k-1}\cup\Omega_k\cup\Omega_{k+1}}|v-c|\lesssim \|\wt N(\nabla v)\|_{L^1(\Delta_0)}$ (obtained as in \eqref{zmez}, using \eqref{eq:Exponential Decay} with the trivial factor $e^{-\beta k}\le 1$), followed by boundary Caccioppoli's inequality and the same splitting of the cones into boundary and interior parts as above --- with Theorem \ref{thm.NtoLoc} applied to the covering cylinders $J_{1/4}(X_i,t_i)$, whose enlargements $J_{1/2}(X_i,t_i)$ span times in $(k-\frac38,k+\frac98]\subset(0,\infty)$ also for $k=1,2,3$, so that the Neumann data of $v$ vanish there --- give the stated bound exactly as for $k>3$.
\qed\medskip

With this in hand we can estimate the norm of 
$\|\wt N(\nabla u) \|_{L^1(\partial\mathcal O\times (4r^2,\infty))}$. We apply 
Lemma \ref{lemma:Exponential Decay} to the solution u of $Lu=0$ whose Neumann data
vanish for all $t>r^2$. We fix an integer $i\in\N$ such that $r^2<i<i+1<4r^2$. Such an integer exists as $r\ge 1$.
By Lemma \ref{lemma:Exponential Decay} it follows that for some $\alpha>0$
$$ \| {\wt N}(\nabla u) \|_{L^1(\Delta_{i+k})}
        \lesssim e^{-\alpha k} \| {\wt N}(\nabla u) \|_{L^1(\Delta_{i})}\le e^{-\alpha k} 
        \|\wt N(\nabla u) \|_{L^1(\partial\mathcal O\times (-4r^2,4r^2))}\lesssim e^{-\alpha k},$$
by \eqref{eq5.5}.    Hence    
$$\|\wt N(\nabla u) \|_{L^1(\partial\mathcal O\times (4r^2,\infty))}
\le \sum_{k=1}^\infty \| {\wt N}(\nabla u) \|_{L^1(\Delta_{i+k})}
\lesssim \sum_{k=1}^\infty  e^{-\alpha k} \le C.$$       
Finally, by combining the estimate above with  \eqref{eq5.5} we get that  $\|\wt N(\nabla u) \|_{L^1(\pom)}\le C$
as desired.\medskip

\subsection{Small atoms, part 1}

We now consider the case $0<r<1$. In what follows $\kappa:=10\sqrt{n+1}$ denotes a fixed dimensional constant (any larger fixed value would also work). Without loss of generality we may assume that our atom supported in
$Q_r(0,0)\cap \partial\Omega$ lies in one of the sets $\Z_j\times (-r^2,r^2)$ for some $j\in\{1,2,\dots,N\}$.
Indeed, this is automatic, if $r<c_n$ for a small dimensional constant $c_n>0$ (chosen so that the spatial diameter of $Q_r(0)$ is less than $1/2$ and that $c_n\le \kappa^{-2}$) since then we have an implication
$$(Q_r(0)\cap \partial\Omega)\cap (1/2)\Z_j\ne\emptyset \Longrightarrow 
Q_r(0)\cap \partial\Omega\subset \Z_j.$$
Obviously, as the union of the sets $(1/2)\Z_j$ covers $\partial\mathcal O$, there is at least one $j$ for which the lefthand side of the implication above is true. Hence, we may assume that $r<c_n$ (if $r\in (c_n,1)$ then an atom
supported in $Q_r(0,0)$ can be written as a finite linear combination of atoms with smaller supports of 
size $r/2$ and therefore considering $r<c_n$ suffices). \medskip

We decompose our atom further. For some integer $K>1$ to be specified later we write $g$ supported in 
$Q_r(0,0)\cap \partial\Omega$ as 
$$
g=g_0+g_1+\dots+g_{2K-2},\quad\mbox{such that supp }g_i\subset Q_r(0)\times (-r^2+ir^2/K,-r^2+(i+2)r^2/K),
$$
$$\mbox{and }\int g_i=0,\qquad \|g_i\|_{L^\infty}\le C(K) \|g\|_{L^\infty}\qquad\mbox{for all }i=0,1,2,\dots,2K-2.$$
Again this is trivial as supports of functions $g_i$, $g_{i+1}$ overlap and hence we can arrange that the average of each function is zero. Thus up to a constant multiple $C(K)$ each $g_i$ is again an atom
and after translation we may therefore from now on assume that we have an atom $g$ with the following properties:
$$\mbox{supp }g\subset[ Q_r(0)\times (-r^2/K,r^2/K)]\cap\pom\subset Q_r(0,0),\quad\int g=0,\quad \|g\|_{L^\infty}\le C(K)r^{-n-1},$$
for some integer $K>1$ to be specified later.\medskip

We start by using the fact that \Np$^{L}$ is solvable. A calculation
similar to \eqref{eq5.5} yields
\begin{multline}\label{eq5.5a}
\|\wt N(\nabla u) \|_{L^1(Q_{8\kappa r}(0,0)\cap\partial\Omega)}\le |Q_{8\kappa r}(0,0)\cap\partial\Omega|^{1/p'}\left(\int_{Q_{8\kappa r}(0,0)\cap\partial\Omega} \wt N(\nabla u)^p d\sigma\right)^{1/p}\\
\lesssim  |Q_{8\kappa r}(0,0)\cap\partial\Omega|^{1/p'} \left(\int_{\partial\Omega}|g|^p\right)^{1/p}\\\le  
|Q_{8\kappa r}(0,0)\cap\partial\Omega|^{1/p'}|Q_{r}(0,0)\cap\partial\Omega|^{1/p}|Q_{r}(0,0)\cap\partial\Omega|^{-1}
\le C, 
\end{multline}
since  $|Q_{8\kappa r}(0,0)\cap\partial\Omega|\sim |Q_{r}(0,0)\cap\partial\Omega|$, as the measure on $\pom$ is doubling.     

\subsection{Unbounded domain related to $\Omega$} \label{ssUD}

Recall that we have that 
$$\mbox{supp $g\subset Q_r(0)\times (-r^2/K,r^2/K)\subset [\Z_j\times (-r^2/K,r^2/K)]\cap \pom$ for some $j\in\{1,2,\dots,N\}$}.$$
Without loss of generality assume that $j=1$ (by relabelling the indices if necessary). It then follows that inside this coordinate patch we can describe the boundary $\partial\mathcal O\times\R$ as follows. 
On $8\Z_1$ there exists
$\phi:\{x'\in\R^{n-1}:\,|x'|\le 8\}\to \R$  a Lipschitz function with Lipschitz constant $\ell$ such that $\phi(0)=0$
and in the coordinates $(x',x_n,t)$ the portion of domain $\Omega$ that belongs to $8\Z_1\times\R$ can be written
as
$$\Omega\cap (8\Z_1\times\R)=\{(x',x_n,t):\,|x'|\le 8,\, \phi(x')<x_n<8(\ell+1),\,t\in\R\},$$
with $\partial\Omega\cap( 8\Z_1\times\R)$ being the graph of function $\phi$ times $\R$.
Let us now choose a new function $\tilde{\phi}:\R^{n-1}\to \R$ with Lipschitz constant at most $C\ell$ (for some absolute constant $C$) having the following properties:

$$\tilde{\phi}(x')={\phi}(x')\quad\mbox{for all }|x'|\le 6, \qquad\mbox{and}\qquad \tilde{\phi}(x')=0\quad\mbox{for all }|x'|\ge 8.$$
This can be done trivially by multiplying $\phi$ by a cutoff function.\medskip

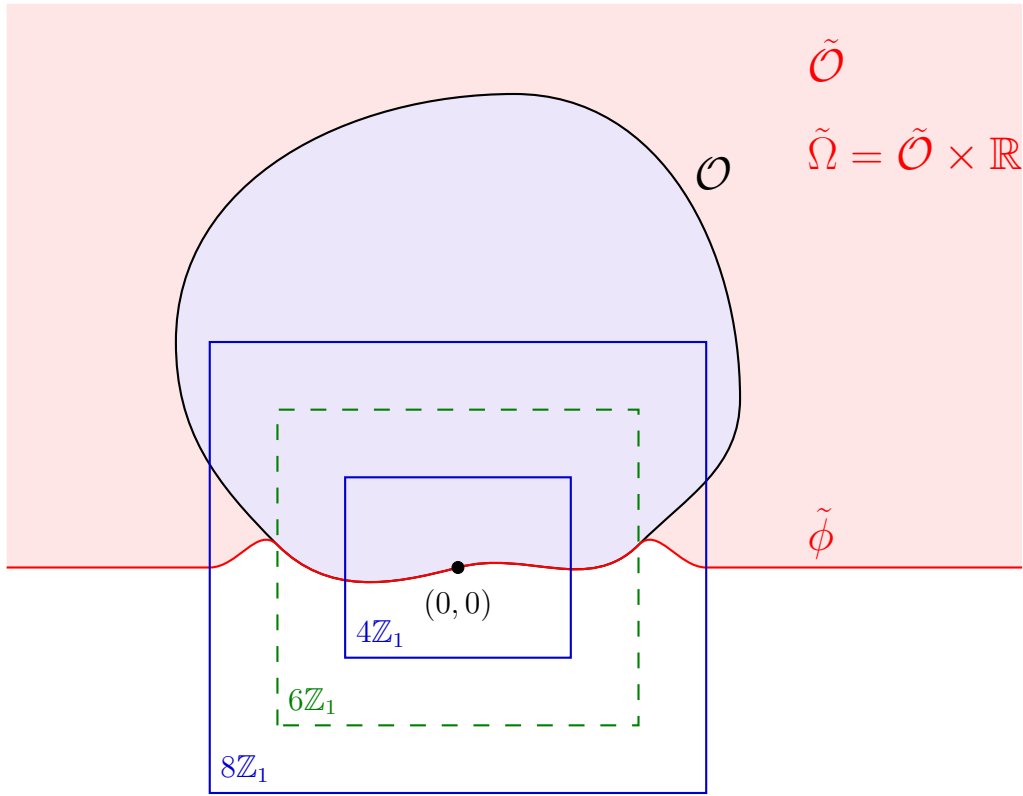
\begin{figure}[htbp]
    \centering
    \begin{tikzpicture}[scale=1.5]

        % 1. Shading for EVERYTHING above the red curve (Reddish region)
        \fill[red!10] 
            (-4.0, 0) -- (-2.2, 0)         
            to[out=0, in=135] (-1.6, 0.2)    % Transition now happens at the edge of 6z_1
            to[out=-45, in=195] (0, 0) to[out=15, in=225] (1.6, 0.2) % Shared segment widened to 6z_1
            to[out=45, in=180] (2.2, 0)    
            -- (5.0, 0) -- (5.0, 5.0) -- (-4.0, 5.0) -- cycle;

        % 2. Black irregular boundary (sigma) with Shading for \tilde{\Omega}
        \filldraw[fill=blue!10, fill opacity=0.8, draw=black, thick, draw opacity=1] 
            (-1.6, 0.2) to[out=-45, in=195] (0, 0) to[out=15, in=225] (1.6, 0.2) % Widened shared segment
            to[out=45, in=-90] (2.5, 1.5)  
            to[out=90, in=0] (0.5, 4.2)    
            to[out=180, in=90] (-2.5, 2)   
            to[out=-90, in=135] cycle;     

        % 3. Red axis 
        \draw[red, thick] 
            (-4.0, 0) -- (-2.2, 0)         
            to[out=0, in=135] (-1.6, 0.2)    
            to[out=-45, in=195] (0, 0) to[out=15, in=225] (1.6, 0.2) 
            to[out=45, in=180] (2.2, 0)    
            -- (5.0, 0);                   

        % 4. Labels 
        \node[font=\Large, anchor=west] at (2.0, 3.5) {$\mathcal O$};
        \node[red, font=\Large, anchor=west] at (3.0, 4.5) {$\tilde{\mathcal O}$};
        \node[red, font=\Large, anchor=west] at (3.0, 0.3) {$\tilde \phi$};
        \node[red, font=\Large, anchor=west] at (3.0, 3.7) {$\tilde{\Omega} = \tilde{\mathcal O} \times \mathbb{R}$};

        % 5. Nested Rectangles 
        \draw[blue!80!black, thick] (-2.2, -2.0) rectangle (2.2, 2.0);
        \node[blue!80!black, anchor=south west] at (-2.2, -2.0) {$8{\mathbb Z}_1$};

        \draw[green!50!black, thick, dashed, dash pattern=on 6pt off 6pt] (-1.6, -1.4) rectangle (1.6, 1.4);
        \node[green!50!black, anchor=south west] at (-1.6, -1.4) {$6{\mathbb Z}_1$};

        \draw[blue!80!black, thick] (-1.0, -0.8) rectangle (1.0, 0.8);
        \node[blue!80!black, anchor=south west] at (-1.0, -0.8) {$4{\mathbb Z}_1$};

        % 6. Origin point
        \filldraw[black] (0,0) circle (1.5pt);
        \node[anchor=north] at (0, -0.1) {$(0,0)$};

    \end{tikzpicture}
    \caption{Visualization of the domains $\mathcal O$ and $\tilde{\Omega}=\tilde{\mathcal O}\times\R$.}
    \label{fig:domain_visualization}
\end{figure}

We similarly extend the coefficients of our equation. The original parabolic PDE for a solution $u$ to $L u=0$ in $\Omega=\mathcal O\times\R$ can be written in the above coordinates as 
$$-\partial_tu+\div(A\nabla u)=0,\qquad\mbox{where $A$ is uniformly elliptic in }[8\Z_1\cap\mathcal O]\times \R.$$\medskip

We would like to extend the PDE to the domain $\tilde\Omega:=\{(x',x_n,t):\,x'\in\R^{n-1},\, x_n>\tilde\phi(x'),\,t\in\R\}$ such that 
the PDE will coincide with the original PDE for $u$ in the set $[4\Z_1\cap\mathcal O]\times \R$, and so that the new matrix $\tilde{A}$
will be uniformly elliptic.
This again is trivial, as we may consider a partition of unity $(\eta_1,\eta_2)$ subordinate to the 
cover of $\{(x',x_n):\,x'\in\R^{n-1},\, x_n>\tilde\phi(x')\}$ by the sets $6\Z_1\cap\mathcal O$ and $\R^n\setminus [4\Z_1\cap\mathcal O]$ with new matrix $\tilde{A}$ defined by
$$\eta_1(x',x_n)A(x',x_n,t)+\eta_2(x',x_n)I,\qquad\mbox{for all } \{(x',x_n,t):\,x'\in\R^{n-1},\, x_n>\tilde\phi(x'),\,t\in\R\}.$$
When $|x'|\le 4$ and $x_n<4(\ell+1)$ this preserves the property that $\tilde A=A$, while for $|x'|\ge 8$ or $x_n>8(\ell+1)$ the matrix is simply the identity matrix and thus trivially satisfies the ellipticity condition. 

It follows that a solution $u$ to the PDE $Lu=0$ in $\Omega$ with Neumann datum $g$
can also be seen as a solution of the new PDE with matrix $\tilde{A}$ on the set $[4\Z_1\cap\mathcal O]\times \R\subset \tilde{\Omega}$.\medskip

From the construction above, clearly $\tilde\Omega$ is an unbounded domain above the graph of a Lipschitz function $\tilde\phi$. We now solve a related Neumann problem on this new domain. 

\[ \begin{cases}
    -\partial_tw+\divg(\tilde A\nabla w) = 0, &\text{in } \tilde \Omega,
    \\
    \partial_\nu^{\tilde A} w= g&\text{on } \partial\tilde \Omega.
\end{cases} \]

We first record the basic a priori regularity of $w$: since
$g$ is a bounded compactly supported function, it belongs to the space $\Hdot^{-1/4}_{\pd_{t} - \Delta_x}(\partial\tilde\Omega)$, and therefore $w\in\dot\E(\tilde\Omega)$.
Our main objective is to establish the decay estimate \eqref{decay2} for $w$ (in this decay estimate $\tilde w$ is just an extension of $w$). It is not clear whether $u$ also enjoys such an estimate; this is one of the issues that make the Neumann problem significantly harder than the analogous Dirichlet problem, where it is possible to compare Green's functions defined on different domains (thanks to the usual comparison principle for solutions). Clearly, we cannot do the same for Neumann functions. Instead, we prove good bounds for $w$ and then we relate them to $u$ using Theorem \ref{thm.NtoLoc}. This theorem is our main tool in many arguments below. \medskip

As we have just said, the solutions $u$ and $w$ are related as they solve the same PDE on the portion of the domain
$[4\Z_1\cap\mathcal O]\times \R$ and also $L(u-w)=0$ there. Furthermore 
$\partial_\nu^{A} (u-w)=g-g= 0\,\text{on } [4\Z_1\cap\partial\mathcal O]\times \R$.

Let us rescale both PDEs (using parabolic dilation) so that now the support of $g$ is contained in the set
$Q_1(0)\times (-1/K,1/K)$. After this rescaling the Neumann datum still has average zero, but 
the size condition changes to $\|g\|_{L^\infty}\le C(K)$. We will refer to rescaled domains and PDEs as $r^{-1}\Omega$ and $r^{-1}\tilde\Omega$ and denote the two solutions we consider (slightly imprecisely) still by $u$ and $w$.
After this rescaling these two PDE coincide on the set 
$[4r^{-1}\Z_1\cap r^{-1}\mathcal O]\times \R$. Observe as $r$ gets smaller this set becomes larger.

Consider a backward parabolic cylinder $J_1=B_{\kappa}(0)\times (2/K-\kappa^2,2/K)$ and hence
$2J_1=J_2=B_{2\kappa}(0)\times (2/K-4\kappa^2,2/K)$; note that $B_{2\kappa}(0)\subset 4r^{-1}\Z_1$ since $r<c_n\le \kappa^{-2}$, so the PDEs for $u$ and $w$ coincide on $J_2\cap r^{-1}\Omega$. We apply Theorem \ref{thm.NtoLoc} for the solution $u-w$
on the set $J_2\cap r^{-1}\Omega$, where its Neumann boundary data vanish. The theorem clearly applies, as the solvability of the Neumann and Dirichlet problems on the domain $\Omega$ implies solvability of the rescaled boundary value problems on the domain $r^{-1}\Omega$.
Hence
\begin{equation*} 
\|\wt N(\nabla (u-w)\1_{J_1})\|_{L^p(\partial (r^{-1}\Omega))} \leq C \iint_{J_{2}\cap r^{-1}\om} |\nabla (u-w)| dXdt.
\end{equation*}
Also recall that by \eqref{eq5.5a} we have after rescaling that 
$$\|\wt N(\nabla u) \|_{L^1(Q_{8\kappa}(0,0)\cap\partial(r^{-1}\Omega))}\le C_K.$$
Using this and the equation above it then follows that 
\begin{equation*} 
\|\wt N(\nabla w\1_{J_1})\|_{L^1(\partial (r^{-1}\Omega))} \leq C_K+C \iint_{J_{2}\cap r^{-1}\om} |\nabla w| dXdt,
\end{equation*}
since on a set of size $O(1)$ clearly $\|\cdot\|_{L^1}\le C\|\cdot\|_{L^p}$ with $C\sim O(1)$.

Consider 
$$\wt{r^{-1}\Omega}_\varepsilon=\{(x,x_n,t):x_n>r^{-1}\tilde\phi(rx)+\varepsilon,\,t\in\R\},$$
denoting the set of points whose distance to $\partial(r^{-1}\tilde\Omega)$ is at least comparable to $\varepsilon>0$.
We make two important observations. First is that $\nabla w=0$ for $t\le-1/K$ which follows from the fact that the support of $g$ is contained in the set $t>-1/K$ and therefore $w$ must be constant for $t\le-1/K$. Without loss of generality we assume that $w=0$ for $t\le-1/K$. 
The second observation is that 
\begin{equation}\label{zex}
 \iint_{J_{2}\cap (r^{-1}\om\setminus \wt{r^{-1}\Omega}_\varepsilon)}|\nabla w| dXdt
\le \varepsilon \|\wt N(\nabla w\1_{J_2})\|_{L^1(\partial (r^{-1}\Omega))}.
\end{equation}
Below we shall prove that $\|\wt N(\nabla w\1_{B_{\kappa}(0)\times (-\infty,2/K)})\|_{L^1(\partial (r^{-1}\tilde\Omega))}$ controls $\wt N$ on a much larger set (c.f. \eqref{eq5.8}) and in particular
$$\|\wt N(\nabla w\1_{J_2})\|_{L^1(\partial (r^{-1}\tilde\Omega))}\le c_0(n,\lambda,\Lambda,\mathcal O)
\|\wt N(\nabla w\1_{B_{\kappa}(0)\times (-\infty,2/K)})\|_{L^1(\partial (r^{-1}\tilde\Omega))}.$$
This together with \eqref{zex} allows us to conclude that
\begin{equation}\label{B1}
\|\wt N(\nabla w\1_{B_{\kappa}(0)\times (-\infty,2/K)})\|_{L^1(\partial (r^{-1}\tilde\Omega))} \leq 2c_0C_K+2c_0C \iint_{(B_{2\kappa}(0)\times (-1/K,2/K))\cap \wt{r^{-1}\Omega}_\varepsilon} |\nabla w| dXdt,
\end{equation}
for sufficiently small $\varepsilon>0$. Here the choice of $\varepsilon>0$ is independent of the scale $r>0$.
We fix $\varepsilon>0$ from now on such that the estimate above holds.
\medskip

Following Brown \cite{B} (see also \cite{DLP3} where the same argument is used), and because
 $\partial^{\tilde A}_\nu w=0$ on the portion of boundary $\partial(r^{-1}\tilde\Omega)\setminus \overline{Q_1(0,0)}$,
we may consider an even reflection $\tilde w$ of $w$ across the boundary $\partial(r^{-1}\tilde\Omega)$. Let us make this reflection, and the associated reflected coefficients, fully explicit. Write $\Phi(x):=r^{-1}\tilde\phi(rx)$ for $x\in\R^{n-1}$; then $\Phi$ is Lipschitz with the same constant as $\tilde\phi$ and
$$r^{-1}\tilde\Omega=\{(x,x_n,t):\,x_n>\Phi(x)\},\qquad \partial(r^{-1}\tilde\Omega)=\{(x,x_n,t):\,x_n=\Phi(x)\}.$$
Consider the reflection map across this graph,
\begin{equation}\label{eqdef.reflec}
\rho(x,x_n,t):=(x,\,2\Phi(x)-x_n,\,t),
\end{equation}
which is a bi-Lipschitz involution of $\R^{n+1}$ (i.e. $\rho\circ\rho=\mathrm{id}$) that fixes $\partial(r^{-1}\tilde\Omega)$ pointwise and interchanges $r^{-1}\tilde\Omega$ with its complement. In the coordinates $(x,x_n,t)$ its Jacobian matrix is
$$D\rho=\begin{pmatrix} I_{n-1} & 0 & 0\\ 2\,\nabla\Phi(x)^{T} & -1 & 0\\ 0 & 0 & 1\end{pmatrix},\qquad |\det D\rho|=1,\quad \|D\rho\|_\infty\le 1+2\|\nabla\Phi\|_\infty.$$
We now define
\begin{equation}\label{eqdef.reflecw}
\tilde w(x,x_n,t):=\begin{cases}
w(x,x_n,t),&\quad x_n\ge \Phi(x),\\[1mm]
w\big(\rho(x,x_n,t)\big)=w(x,\,2\Phi(x)-x_n,\,t),&\quad x_n< \Phi(x),\end{cases}
\end{equation}
which extends $w$ to the set $\R^{n+1}\setminus \overline{Q_1(0,0)}$, together with the reflected coefficient matrix
\begin{equation}\label{eqdef.reflecA}
\dbtilde A:=\begin{cases}
\tilde A,&\quad\text{on }r^{-1}\tilde\Omega,\\[1mm]
(D\rho)\,\big(\tilde A\circ\rho\big)\,(D\rho)^{T},&\quad\text{on }\R^{n+1}\setminus r^{-1}\tilde\Omega.\end{cases}
\end{equation}
Since $D\rho$ is bounded and invertible with $|\det D\rho|=1$, and $\tilde A$ is bounded and uniformly elliptic, the matrix $\dbtilde A$ is again bounded and uniformly elliptic on $\R^{n+1}\setminus \overline{Q_1(0,0)}$, with constants depending only on $\lambda$, $\Lambda$ and the Lipschitz constant $\ell$; indeed for the lower bound $\langle\dbtilde A\,\xi,\xi\rangle=\langle(\tilde A\circ\rho)\,(D\rho)^{T}\xi,(D\rho)^{T}\xi\rangle\ge \lambda\,|(D\rho)^{T}\xi|^{2}\gtrsim_{\ell}|\xi|^{2}$. A direct change of variables in the weak formulation, using that the co-normal derivative $\partial^{\tilde A}_\nu w$ vanishes on $\partial(r^{-1}\tilde\Omega)\setminus\overline{Q_1(0,0)}$, shows that $\tilde w$ is a weak solution of the parabolic PDE
$$\partial_t\tilde w-\divg(\dbtilde A\nabla\tilde w)=0$$
across that portion of the boundary; this is precisely the reflection principle of \cite{B}. Finally, the reflection preserves the energy class \emph{trivially}: since $\rho$ acts only on the spatial variable $x_n$ and leaves the time variable $t$ untouched, while $|\det D\rho|=1$, one has for every space-time region $S\subset r^{-1}\tilde\Omega$
$$\int_{\rho(S)}|\nabla\tilde w|^2\,dXdt\le \|D\rho\|_\infty^{2}\int_{S}|\nabla w|^2\,dXdt,\qquad \|\HT\dhalf\tilde w\|_{L^2(\rho(S))}=\|\HT\dhalf w\|_{L^2(S)},$$
the last identity because $\HT$ and $\dhalf$ act in the time variable only, which $\rho$ does not move. Hence $\tilde w$ inherits the local energy bounds of $w$ and is again a (local) energy solution on $\R^{n+1}\setminus\overline{Q_1(0,0)}$.

We first establish that $\tilde w$ is bounded away from $Q_{1}(0,0)$ when $t<2/K$.  The argument here is as in \cite{DLP3} but fixes a small issue the argument in \cite{DLP3} contained.
Let  $(X,t)$ be any point such that $J_{1/4}(X,t)\cap Q_{1}(0,0)=\emptyset$, $\|(X,t)\|\le \kappa/4$ and $t<2/K$.  We claim that then we have
\begin{equation}\label{eq5.3}
\fiint_{J_{1/8}(X,t)}|\tilde w|\le C \|\wt N(\nabla w\1_{B_{\kappa}(0)\times (-\infty,2/K)})\|_{L^1(\partial (r^{-1}\tilde\Omega))},
\end{equation}

and hence by interior H\"older regularity of $\tilde w$
\begin{equation}\label{eq5.3a}
\sup_{J_{1/16}(X,t)}|\tilde w| \le C \|\wt N(\nabla w\1_{B_{\kappa}(0)\times (-\infty,2/K)})\|_{L^1(\partial (r^{-1}\tilde\Omega))}.
\end{equation}
To see \eqref{eq5.3} consider first points with $\|(X,t)\|\le \kappa/2$ and $t\le 2/K$. For such points every $(Y,s)\in J_{1/4}(X,t)\cap r^{-1}\tilde\Omega$ lies in the cone $\Gamma_a(q,s)$ of the boundary point $(q,s)\in\Delta_{2\kappa}(0,0)$ realizing $\delta(Y,s)$ (for points on the reflected side one first applies the reflection $\rho$, at the cost of a constant depending on the Lipschitz character), and the corresponding averaging ball is contained in $B_{\kappa}(0)\times(-\infty,2/K)$. Decomposing $J_{1/4}(X,t)$ into Whitney-type regions, each of which is seen by a set of boundary points of $\Delta_{2\kappa}(0,0)$ of comparable measure on which $\wt N(\nabla w \1_{B_{\kappa}(0)\times (-\infty,2/K)})$ dominates the corresponding averages of $|\nabla\tilde w|$, we obtain
\begin{equation}\label{eq5.4}
\fiint_{J_{1/4}(X,t)}|\nabla \tilde w|\lesssim \int_{\Delta_{2\kappa}(0,0)}\wt N(\nabla w \1_{B_{\kappa}(0)\times (-\infty,2/K)})\,d\sigma\le C \|\wt N(\nabla w\1_{B_{\kappa}(0)\times (-\infty,2/K)})\|_{L^1(\partial (r^{-1}\tilde\Omega))}.
\end{equation}
Then by interior Poincar\'e's inequality for $\tilde w$ (Lemma \ref{lem.Pnc_sol})
$$\fiint_{J_{1/8}(X,t)}|\tilde w-c_{\tilde{w}}|\le C \|\wt N(\nabla w\1_{B_{\kappa}(0)\times (-\infty,2/K)})\|_{L^1(\partial (r^{-1}\tilde\Omega))},$$
where $c_{\tilde{w}}$ denotes the average of $\tilde{w}$ in $J_{1/8}(X,t)$. Clearly, if we can take $c_{\tilde{w}}=0$ the claim would follow. We know however that this holds for $t<-1/K$ as $\tilde{w}$ is zero there.

To remove the constant $c_{\tilde w}$ we chain the estimate along a sequence of overlapping regions (see also \cite{DLP3} and the discussion following (5.5) there), which we reproduce here for completeness. Simple geometric considerations show that there exist $N=N(n)$ and a chain of regions $J_{1/4}(X_1,t_1),J_{1/4}(X_2,t_2),\dots,J_{1/4}(X_k,t_k)$ with $k\le N$ such that
\begin{itemize}
\item the consecutive regions $J_{1/4}(X_i,t_i)$ and $J_{1/4}(X_{i+1},t_{i+1})$ overlap for $i=1,\dots,k-1$;
\item $J_{1/4}(X_i,t_i)\cap Q_1(0,0)=\emptyset$ and $J_{1/4}(X_i,t_i)\subset\{t\le 2/K\}$ for all $i$;
\item $\|(X_i,t_i)\|\le\kappa/2$ for all $i$, so that \eqref{eq5.4} applies to each $J_{1/4}(X_i,t_i)$;
\item for each $i$ either $t_i=t_{i+1}$ or $X_i=X_{i+1}$;
\item $t_1<-1/K$ and $(X_k,t_k)=(X,t)$.
\end{itemize}
The last condition forces the average of $\tilde w$ over $J_{1/8}(X_1,t_1)$ to vanish, because $\tilde w\equiv 0$ for $t<-1/K$. It therefore suffices to bound the difference of the averages of $\tilde w$ over consecutive $J_{1/8}(X_i,t_i)$ and $J_{1/8}(X_{i+1},t_{i+1})$ by the right-hand side of \eqref{eq5.4}. When $t_i=t_{i+1}$ this is immediate: expressing the difference of the two values of $\tilde w$ at points with equal time coordinate by the fundamental theorem of calculus, integrated along the horizontal segment joining them, bounds it by the average of $|\nabla\tilde w|$ over the convex hull of the two regions, hence by \eqref{eq5.4}. When $X_i=X_{i+1}$ but $t_i\ne t_{i+1}$ one instead uses the equation satisfied by $\tilde w$ to trade the increment in time for a spatial gradient; this is carried out in full detail in \cite{DLP3}, and rests on the same mechanism as the interior Poincar\'e inequality for solutions (Lemma \ref{lem.Pnc_sol}), where an increment in time is likewise controlled by a spatial gradient by means of the equation. Chaining these at most $N$ differences of averages, and using that $N$ depends only on $n$, we obtain $\fiint_{J_{1/8}(X,t)}|\tilde w|\le CN\,\|\wt N(\nabla w\1_{B_{\kappa}(0)\times (-\infty,2/K)})\|_{L^1(\partial (r^{-1}\tilde\Omega))}$, which is \eqref{eq5.3}. Hence \eqref{eq5.3} holds for all points $(X,t)$ outside of $Q_1(0,0)$
with $\|(X,t)\|\le \kappa/4$ and $t\le 2/K$.

Boundedness for the remaining points (i.e. those with $\|(X,t)\|\ge \kappa/4$) follows by the maximum principle for the parabolic operator $\partial_t-\divg(\dbtilde A\nabla\cdot)$ with bounded measurable coefficients (see e.g. \cite[Ch.~2]{Lieberman}), applied on the region $\{\|(X,s)\|>\kappa/4\}$. Indeed, since $Q_1(0,0)\subset\{\|(X,s)\|<\kappa/4\}$, on this region $\tilde w$ (defined on both sides of the graph by the reflection) solves the PDE, it vanishes for $s<-1/K$, and its values on the lateral boundary $\{\|(X,s)\|=\kappa/4\}$ are bounded by the right-hand side of \eqref{eq5.3a}. Some care is needed because the region is unbounded, so that the classical comparison principle does not apply directly; however $\tilde w\in\dot\E$ locally, and the finiteness of the energy forces the averaged decay $\fiint_{J_1(X,t)}|\tilde w|^2\to0$ as $|X|\to\infty$ (uniformly for $t$ in compact sets; this follows by chaining backwards in time to $\{t<-1/K\}$ exactly as in the proof of \eqref{eq5.3}, with the right-hand side now the tail energy of $\nabla w$), which rules out growth at infinity and legitimizes the comparison principle through a standard Phragm\'en--Lindel\"of exhaustion by bounded subdomains (cf. the growth-restricted maximum principles in \cite[Ch.~2]{Lieberman}). Hence
$$|\tilde{w}(X,t)|\le C\|\wt N(\nabla w\1_{B_{\kappa}(0)\times (-\infty,2/K)})\|_{L^1(\partial (r^{-1}\tilde\Omega))}$$
on the complement of any enlargement of $Q_1(0,0)$ with $t<2/K$.\medskip

Consider now a smooth cutoff function $\eta$ which vanishes on $Q_1(0,0)$ and is equal to $1$ outside of $Q_2(0,0)$. It follows that $\eta\tilde{w}$ is bounded on $\R^n\times (-\infty,2/K)$ and hence for any $c>0$ and any bounded time interval $(a,b)$, $b\le 2/K$ we have that
$$\int_a^b\int_{\R^n}|(\eta\tilde{w})(X,t)|^2e^{-c|X|^2}dX\,dt<\infty,$$
which then implies the integrability also on the interval $(-\infty,2/K)$ as $\tilde w$ vanishes when $t<-1/K$.
Recall  Aronson’s bound for the fundamental 
solution (\cite{Aro68}), denoted by $E(X,t,Y,s)$:
\begin{equation}\label{KE}
|E(X,t,Y,s)|\le C\chi_{(0,\infty)}(t-s)(t-s)^{-n/2}e^{-c|X-Y|^2/(t-s)},
\end{equation}
for some $c>0$. Here $E$ is the fundamental solution of the operator $\partial_t-\divg(\dbtilde A\nabla\cdot)$, where the matrix $\dbtilde A$ is extended inside $\overline{Q_1(0,0)}$ by the identity (the choice of this extension is immaterial below, as $\eta\tilde w$ vanishes there). Also $\int_{\R^n} E(X,t,Y,s)dY=1$ for all $(X,t)\in\R^n\times\R$ and $s<t$.

It follows that $(\eta\tilde w)$ can be written as
\begin{multline*}
\eta\tilde w(X,t)=\iint_{\R^n\times\R}E(X,t,Y,s)\left[\tilde w(Y,s)\partial_s\eta(Y,s)-\langle \dbtilde A(Y,s)\nabla\tilde w(Y,s),\nabla\eta(Y,s)\rangle \right]dYds\\+\iint_{\R^n\times\R}
\tilde w(Y,s)\langle \dbtilde A(Y,s)\nabla \eta(Y,s),\nabla_YE(X,t,Y,s)\rangle dYds.
\end{multline*}
Hence using \eqref{KE} and Caccioppoli's inequality for the term involving $\nabla E$, it follows (recalling that $\eta\tilde w=\tilde w$ outside $Q_2(0,0)$) that for some small $c>0$
\begin{equation}\label{decay}
|\tilde w(X,t)|\le Ce^{-c|X|^2},\qquad \mbox{for all $(X,t)\notin Q_2(0,0)$ with }t<2/K.
\end{equation}
Here the constant $C$ in the above estimate is a multiple of $\|\wt N(\nabla w\1_{B_{\kappa}(0)\times (-\infty,2/K)})\|_{L^1(\partial (r^{-1}\tilde\Omega))}$. The decay we have picked up in \eqref{decay} is crucial. 
It enables us to estimate $\|\wt N(\nabla w \1_{t<2/K}) \|_{L^1(S_k)}$ on a sequence of
annuli $S_k$ partitioning the complement of the set $Q_{1}(0,0)\cap\partial(r^{-1}\tilde\Omega)$.
$$S_k=\{(X,t)\in \partial(r^{-1}\tilde\Omega): 2^k<\|(X,t)\|\le 2^{k+1}\},\qquad\mbox{for } k=1,2,\dots.$$

We omit the proof here, since \eqref{decay} provides exponential (Gaussian) decay, which is faster than the polynomial one; the analogous but more delicate estimate under the weaker polynomial decay \eqref{decay2} (of order $-n-\gamma$) is carried out in full detail below (see the calculation leading to \eqref{Est-w2}).
The only difference from that calculation is that here we stop at the scale $2^{k_0}\sim 2r^{-1}$, since the rescaled PDEs for $w$
and $u$ coincide up to the scale $[4r^{-1}\Z_1\cap r^{-1}\mathcal O]\times \R$. It then follows:

\begin{multline}\label{eq5.8}
\|\wt N(\nabla w\1_{t<2/K}) \|_{L^1((2r^{-1}\Z_1\times (-2r^{-2},2r^{-2}))\cap \partial(r^{-1}\tilde\Omega))}\\\le 
\|\wt N(\nabla w\1_{B_{\kappa}(0)\times (-\infty,2/K)})\|_{L^1(\partial (r^{-1}\tilde\Omega))}+\sum_{k=1}^{k_0}
\|\wt N(\nabla w\1_{t<2/K}) \|_{L^1(S_k)}\\
\le C\|\wt N(\nabla w\1_{B_{\kappa}(0)\times (-\infty,2/K)})\|_{L^1(\partial (r^{-1}\tilde\Omega))}\lesssim_{c_0,C}
C_K+ \iint_{(B_{2\kappa}(0)\times (-1/K,2/K))\cap \wt{r^{-1}\Omega}_\varepsilon} |\nabla w| dXdt,
\end{multline}
where the last inequality follows from \eqref{B1}. The implied constants in the estimate above do not depend on the scale $r>0$, and only depend on the ellipticity constants, the Lipschitz constant of the graph and the constants in the estimates for the 
\Np$^{L}$ and \Dq$^{L^*}$ problems. With the exception of the constant $C_K$, the remaining constants do not depend on $K$ (which is not yet chosen). \medskip

Our next objective is to establish estimates for the last term on the righthand side of \eqref{eq5.8}.
For an ease of notation, let 
\begin{align*}
S_\varepsilon&=B_{2\kappa}(0)\times (2/K-4\kappa^2,2/K)\cap \wt{r^{-1}\Omega}_\varepsilon,\\
\wt S_\varepsilon&=B_{2\kappa+1}(0)\times (2/K-(2\kappa+1)^2,2/K)\cap \wt{r^{-1}\Omega}_{\varepsilon/2},\\
\dbtilde S_\varepsilon&=B_{2\kappa+1}(0)\times (2/K-(2\kappa+1)^2,2/K)\cap \wt{r^{-1}\Omega}_{\varepsilon/4},\\
\trtilde S_\varepsilon&=B_{2\kappa+2}(0)\times (2/K-(2\kappa+2)^2,2/K)\cap \wt{r^{-1}\Omega}_{\varepsilon/5}.
\end{align*}
Clearly as $w=0$ for $t<-1/K$, the integral of $\nabla w$ over $S_\varepsilon$ is equal to the last term on the righthand side of \eqref{eq5.8}. We can think about this term as an interior parabolic \lq\lq cube" 
whose distance to the boundary is comparable to $\varepsilon$. Hence by Lemma \ref{L:Caccio} and H\"older
it follows that
\begin{equation}\label{rex}
\iint_{(B_{2\kappa}(0)\times (-1/K,2/K))\cap \wt{r^{-1}\Omega}_{\varepsilon}} |\nabla w| dXdt
\lesssim C\varepsilon^{-1}\left(\iint_{\wt S_\varepsilon}w^2\right)^{1/2}.
\end{equation}
Recall that since $w=0$ for $t<-1/K$, by the fundamental theorem of calculus (integrating in the $t$-variable and using Fubini's theorem)
we have that for a fixed $X$:
$$\int_{-\infty}^{2/K} w^2(X,t)dt= 2\int_{-\infty}^{2/K} (2/K-s)\,\partial_s w(X,s)w(X,s)\, ds,$$
where the weight satisfies $0\le 2/K-s\le 3/K$ on the time interval where $w$ does not vanish, whose length is proportional to $1/K$.
Let $\varphi$ be a smooth nonnegative cutoff function in $X$ variables only (constant in $t$) such that $\varphi=1$ on 
$\wt S_\varepsilon$ and $\varphi$ vanishes outside of $\dbtilde S_\varepsilon$ and $|\nabla\varphi|\lesssim \varepsilon^{-1}$.  Then using the equation above we have that
\begin{multline}\label{rex2}
\iint_{\wt S_\varepsilon}w^2 dXdt \le \iint_{\dbtilde S_\varepsilon}w^2(X,t)\varphi^2(X) dXdt
= 2\iint_{\dbtilde S_\varepsilon}(2/K-t)\,\partial_tw(X,t)w(X,t)\varphi^2(X) dXdt\\
=2\iint_{\dbtilde S_\varepsilon}(2/K-t)\,\divg(\tilde A\nabla w)(X,t)w(X,t)\varphi^2(X) dXdt\\
=-2\iint_{\dbtilde S_\varepsilon}(2/K-t)\,\tilde{A}\nabla w(X,t)\cdot\nabla w(X,t)\varphi^2(X) dXdt\\
-2\iint_{\dbtilde S_\varepsilon}(2/K-t)\, \tilde A\nabla w(X,t)\cdot\nabla(\varphi^2(X))w(X,t)dXdt\\
\le C/K\iint_{\dbtilde S_\varepsilon} |\nabla w||\nabla\varphi||w||\varphi|dXdt
\le 1/2 \iint_{\dbtilde S_\varepsilon}w^2\varphi^2 dXdt+C/(\varepsilon K)^2 \iint_{\dbtilde S_\varepsilon}|\nabla w|^2dX dt.
\end{multline}
Here in the penultimate inequality we have dropped the term containing $\tilde{A}\nabla w\cdot\nabla w\,\varphi^2\ge 0$ (as it appears with a negative sign) and used that $0\le 2/K-t\le 3/K$ on the support of $w$. Hence after hiding the penultimate term we get that
$$\iint_{\wt S_\varepsilon}w^2 dXdt \le \iint_{\dbtilde S_\varepsilon}w^2\varphi^2 dXdt
\le 2C/(\varepsilon K)^2 \iint_{\dbtilde S_\varepsilon}|\nabla w|^2dX dt.$$
We note that to derive this inequality we worked with $\partial_t w$ which for a general parabolic PDE is only a distribution, but the calculation above can be made rigorous by approximating the coefficients by smooth
ones and a limiting argument. Finally, by using reverse H\"older inequality for the gradient of a solution (which is a consequence of Lemma \ref{L:Caccio}) we get that
$$\left(\iint_{\wt S_\varepsilon}w^2 dXdt \right)^{1/2}
\le C_\varepsilon/ K \iint_{\,\,\trtilde S_\varepsilon}|\nabla w|dX dt.$$
By combining the estimate above with \eqref{rex} and \eqref{eq5.8} we finally get that
\begin{multline}\label{eq5.11}
\|\wt N(\nabla w\1_{t<2/K}) \|_{L^1((2r^{-1}\Z_1\times (-2r^{-2},2r^{-2}))\cap \partial(r^{-1}\tilde\Omega))}\\\le
C_K+ C_{\varepsilon}K^{-1}
\iint_{(B_{2\kappa+2}(0)\times (-\infty,2/K))\cap \wt{r^{-1}\Omega}_{\varepsilon/5}} |\nabla w| dXdt.
\end{multline}
The final observation is that since 
$$\iint_{(B_{2\kappa+2}(0)\times (-\infty,2/K))\cap \wt{r^{-1}\Omega}_{\varepsilon/5}} |\nabla w| dXdt
\le C \|\wt N(\nabla w\1_{t<2/K}) \|_{L^1((2r^{-1}\Z_1\times (-2r^{-2},2r^{-2}))\cap \partial(r^{-1}\tilde\Omega))},$$
we see that
\begin{multline}\label{eq5.11a}
\|\wt N(\nabla w\1_{t<2/K}) \|_{L^1((2r^{-1}\Z_1\times (-2r^{-2},2r^{-2}))\cap \partial(r^{-1}\tilde\Omega))}\\\le
C_K+ \tilde C_{\varepsilon}K^{-1}
\|\wt N(\nabla w\1_{t<2/K}) \|_{L^1((2r^{-1}\Z_1\times (-2r^{-2},2r^{-2}))\cap \partial(r^{-1}\tilde\Omega))}.
\end{multline}
The absorbing coefficient in \eqref{eq5.11a} is $\tilde C_{\varepsilon}K^{-1}$, and since $\tilde C_\varepsilon$ may itself be large we must take $K$ correspondingly large, namely $K\ge 2\tilde C_\varepsilon$, so that $\tilde C_\varepsilon K^{-1}\le 1/2$. With such a choice of $K$ we get that
\begin{equation}
\label{Est-w1}
\|\wt N(\nabla w\1_{t<2/K}) \|_{L^1((2r^{-1}\Z_1\times (-2r^{-2},2r^{-2}))\cap \partial(r^{-1}\tilde\Omega))}\le 2C_K.
\end{equation}
This is a \lq\lq near final" version of our estimate for $w$. We still need to take care of estimates for $\nabla w$ for $t\ge 2/K$ to obtain an analogue of the estimate \eqref{Est-w1} without this restriction on the time interval. With \eqref{Est-w1} in hand we revisit our construction we gave for $\tilde w$.\medskip

Recall, that we have constructed $\tilde{w}$ by reflection (excluding the parabolic ball $Q_1(0,0)$).
Our objective previously was to establish bounds that do not depend on $K$ and then choose $K$ sufficiently large so that a key term (as seen above) can be absorbed by the left-hand side. With this done and $K$ fixed permanently we now can observe that the Neumann datum $g$ is in fact supported in $(Q_1(0)\times (-1/K,1/K))\cap\partial\tilde\Omega$. Hence we can do an even reflection of $w$ and obtain a function defined
on the set $\R^{n+1}\setminus \overline{Q_1(0)\times (-1/K,1/K)}$. We shall still call this function $\tilde w$ as it is an extension to a larger set of the previously defined $\tilde w$.

We now claim that our function $\tilde w$ is also bounded for $t>1/K$. To see this we claim that a better version of \eqref{eq5.3} now holds, namely that
\begin{equation}\label{eq5.3b}
\sup_{J_{1/(16K)}(X,t)}|\tilde w|\lesssim \fiint_{J_{1/(8K)}(X,t)}|\tilde w|\le C \|\wt N(\nabla w\1_{t<2/K}) \|_{L^1((2r^{-1}\Z_1\times (-2r^{-2},2r^{-2}))\cap \partial(r^{-1}\tilde\Omega))},
\end{equation}
for all points $(X,t)$ such that  $J_{1/(4K)}(X,t)\cap (Q_{1}(0)\times (-1/K,1/K)))=\emptyset$.  
The proof is analogous to the one given above, where once again we first prove the claim for
the points such that $\|(X,t)\|\le O(\kappa)$ and $t<2/K$, use the maximum principle (justified as before, cf. \cite{Lieberman}) to get it for all points
$(X,t)\in \R^{n+1}\setminus \overline{Q_1(0)\times (-1/K,1/K)}$ with $t<2/K$ and then since the set
$\overline{Q_1(0)\times (-1/K,1/K)}$ does not meet the hyperplane $t=2/K$ we already have boundedness of the solution $\tilde w$ on this hyperplane. Hence, we can use the maximum principle again for further times ($t\ge 2/K$). We obtain that outside of any open enlargement of the set
$\overline{Q_1(0)\times (-1/K,1/K)}$ there holds
$$|\tilde w(X,t)|\le C_K  \|\wt N(\nabla w\1_{t<2/K}) \|_{L^1((2r^{-1}\Z_1\times (-2r^{-2},2r^{-2}))\cap \partial(r^{-1}\tilde\Omega))}\le \tilde C_K.$$
Note the dependence on $K$ which explains the two-step process. Modifying now the cutoff function $\eta$
and again using \eqref{KE} and the representation of $\eta\tilde w$ we obtain an analogue of \eqref{decay}
\begin{equation}\label{decay-global}
|\tilde w(X,t)|\le Ce^{-c|X|^2},\qquad \mbox{for all $(X,t)\notin Q_2(0,0)$ with }t<20,
\end{equation}
for some small $c>0$. The value $20$ has no special significance here: the representation formula via \eqref{KE} yields Gaussian decay in the spatial variable up to any fixed finite time (with constants depending on that time), and any sufficiently large value works for the argument below.

To pass from the Gaussian bound \eqref{decay-global} to the polynomial decay \eqref{decay2} we follow \cite{DLP3}; we reproduce the argument here for the reader's convenience. First we record that $\tilde w$ has vanishing spatial average on a fixed time slice. Fix $t_*\in (1/K,20)$. Using that $w\equiv 0$ for $t<-1/K$, the equation $\partial_t w=\divg(\tilde A\nabla w)$, the divergence theorem in the (time-independent) base $\{x_n>\Phi(x)\}$ of $r^{-1}\tilde\Omega$ (the flux at spatial infinity vanishing thanks to the Gaussian bound \eqref{decay-global}), and $\partial^{\tilde A}_\nu w=g$ with $\int g=0$, we obtain
$$\int_{\{x_n>\Phi(x)\}} w(X,t_*)\,dX=\int_{-1/K}^{t_*}\!\!\int_{\{x_n>\Phi(x)\}}\!\partial_t w\,dXdt=\int_{-1/K}^{t_*}\!\!\int_{\partial (r^{-1}\tilde\Omega)}\!\partial^{\tilde A}_\nu w\,d\sigma\,dt=\int g=0.$$
Since $\tilde w$ is the even reflection \eqref{eqdef.reflecw} of $w$ and $|\det D\rho|=1$, it follows that $\int_{\R^n}\tilde w(X,t_*)\,dX=0$ as well. Moreover, for $t\ge t_*$ the excluded set $\overline{Q_1(0)\times(-1/K,1/K)}$ lies entirely in the past, so $\tilde w$ solves the clean initial value problem $\partial_t\tilde w-\divg(\dbtilde A\nabla\tilde w)=0$ on all of $\R^n\times(t_*,\infty)$ with initial datum $\tilde w(\cdot,t_*)$. The claim \eqref{decay2} for $t\ge t_*$ is therefore a consequence of the following elementary decay lemma, applied after the shift $t\mapsto t-t_*$; for $t\le t_*$ it is already contained in \eqref{decay-global}.

\begin{lemma}\label{lem.decay}
Let $\dbtilde A$ be bounded and uniformly elliptic, and let $f:\R^n\to\R$ satisfy, for some small $\gamma>0$,
$$|f(X)|\le C_0(1+|X|)^{-n-2\gamma}\qquad\text{and}\qquad \int_{\R^n}f(X)\,dX=0.$$
Then the solution $U$ of $\partial_t U-\divg(\dbtilde A\nabla U)=0$ in $\R^n\times(0,\infty)$ with $U(\cdot,0)=f$ satisfies, with a constant $C=C(n,\lambda,\Lambda,\gamma,C_0)$,
$$|U(X,t)|\le C(1+\|(X,t)\|)^{-n-\gamma}\qquad\text{for all }t>0.$$
\end{lemma}

\begin{proof}
We omit the proof; it splits the representation $U(X,t)=\int_{\R^n}E(X,t,Y,0)f(Y)\,dY$ into near-field and far-field parts and combines the vanishing average of $f$ with the Aronson bounds \eqref{KE}. See Lemma~5.7 of \cite{DLP3} for the details (cf.\ also Lemma~3.13 of \cite{B}).
\end{proof}

Applying Lemma~\ref{lem.decay} with $f=\tilde w(\cdot,t_*)$, whose hypotheses hold by \eqref{decay-global} together with the boundedness of $\tilde w(\cdot,t_*)$ on $Q_2(0,0)$ established above (so that $|f|\le C_0(1+|X|)^{-n-2\gamma}$ with $C_0=C_0(K)$) and the zero-average property just established, we conclude that for some small $\gamma>0$
\begin{equation}\label{decay2}
|\tilde w(X,t)|\le C\|(X,t)\|^{-n-\gamma},\qquad \mbox{for all $(X,t)\notin Q_2(0,0)$}.
\end{equation}
With the polynomial decay \eqref{decay2} in hand we can now estimate $\|\wt N(\nabla w)\|_{L^1(S_k)}$ over the same annuli $S_k$ as before, but this time \emph{without} the temporal cutoff $\1_{t<2/K}$. We reproduce the corresponding argument of \cite{DLP3} in the present setting. Close to the datum, H\"older's inequality and the solvability of \Np$^L$ for the rescaled problem give
$$\|\wt N(\nabla w) \|_{L^1(Q_{8\kappa}(0,0)\cap \partial(r^{-1}\tilde\Omega))}\le |Q_{8\kappa}(0,0)\cap\partial(r^{-1}\tilde\Omega)|^{1/p'}\Big(\int|g|^p\Big)^{1/p}\le C_K.$$
For $k$ with $8\kappa\le 2^k\le 2^{k_0}$ (recall we stop at $2^{k_0}\sim 2r^{-1}$; the annuli with $2^k<8\kappa$ lie in $Q_{8\kappa}(0,0)$ and are already covered by the near-field bound) we cover $S_k$ by $N$ boundary cubes $Q^{k,j}_{r_k}\cap\partial(r^{-1}\tilde\Omega)$ of size $r_k=c2^k$, $j=1,\dots,N$, with $N$ independent of $k$ and $d(Q^{k,j}_{8r_k},Q_1(0,0))\gtrsim 2^k$. Since $\partial^{\tilde A}_\nu w=0$ on $8Q^{k,j}_{r_k}\cap\partial(r^{-1}\tilde\Omega)$, Theorem~\ref{thm.NtoLoc} followed by the boundary Caccioppoli inequality (Lemma~\ref{L:Caccio}) gives, for the truncated maximal function $\wt N^{r_k}$ (defined as $\wt N$ but with the non-tangential cones truncated at parabolic height $r_k$),
\begin{multline*}
\|\wt N^{r_k}(\nabla w) \|_{L^p(Q^{k,j}_{r_k}\cap \partial(r^{-1}\tilde\Omega))}\le \|\wt N(\nabla w\1_{2Q^{k,j}_{r_k}}) \|_{L^p(\partial(r^{-1}\tilde\Omega))}\\\le Cr_k^{(n+1)/p}\Big(\fiint_{4Q^{k,j}_{r_k}\cap r^{-1}\tilde\Omega}|\nabla w|^2\Big)^{1/2}\le Cr_k^{-1+(n+1)/p}\Big(\fiint_{8Q^{k,j}_{r_k}\cap r^{-1}\tilde\Omega}|w|^2\Big)^{1/2}.
\end{multline*}
As every point of $8Q^{k,j}_{r_k}$ has $\|(X,t)\|\gtrsim 2^k$, the polynomial bound \eqref{decay2} (recall that $w=\tilde w$ there) gives $|w|\le Cr_k^{-n-\gamma}$, whence
$$\|\wt N^{r_k}(\nabla w) \|_{L^p(Q^{k,j}_{r_k}\cap \partial(r^{-1}\tilde\Omega))}\le Cr_k^{-1+(n+1)/p}\,r_k^{-n-\gamma}=Cr_k^{-(n+1)/p'-\gamma}.$$
Summing over $j$ by H\"older's inequality,
$$\|\wt N^{r_k}(\nabla w) \|_{L^1(S_k)}\le CN\,|Q^{k,j}_{r_k}\cap \partial(r^{-1}\tilde\Omega)|^{1/p'}\|\wt N^{r_k}(\nabla w) \|_{L^p(Q^{k,j}_{r_k}\cap \partial(r^{-1}\tilde\Omega))}\lesssim r_k^{(n+1)/p'}r_k^{-(n+1)/p'-\gamma}\sim r_k^{-\gamma}.$$
The contribution to $\wt N(\nabla w)$ of the interior part of the cones (points $(Y,s)$ with $\delta(Y,s)\ge r_k$) is estimated in exactly the same way, using the interior Caccioppoli inequality and \eqref{decay2}, and is again $O(r_k^{-\gamma})$ on $S_k$. Summing the resulting geometric series over the remaining $k\le k_0$ and adding the near-field bound, we obtain
\begin{equation}
\label{Est-w2}
\|\wt N(\nabla w) \|_{L^1((2r^{-1}\Z_1\times (-2r^{-2},2r^{-2}))\cap \partial(r^{-1}\tilde\Omega))}\le C_K.
\end{equation}
We now undo the parabolic rescaling we performed earlier and shrink the set 
$(2r^{-1}\Z_1\times (-2r^{-2},2r^{-2}))\cap \partial(r^{-1}\tilde\Omega)$ back to the original set $(2\Z_1\times (-2,2))\cap \partial\tilde\Omega$. Hence for our original (non-rescaled $w$) we get that
\begin{equation}\label{Est-w3}
\|\wt N(\nabla w) \|_{L^1((2\Z_1\times (-2,2))\cap \partial\tilde\Omega)}\le C_K.
\end{equation}

\subsection{Small atoms, part 2}

Recall now that $u$ solves $Lu=0$ in $\Omega$ with Neumann data $g$ (same as $w$) and that we have already established \eqref{eq5.5} in the part 1. However, the PDEs for $u$ and $w$ coincide on the set
$(4\Z_1\times\R)\cap \Omega$ and therefore we may apply Theorem \ref{thm.NtoLoc} for the difference $u-w$
on the set $J_2\cap\Omega$, where
$$J_1= \Z_1\times (2/K-1,2/K),\qquad 2J_1=2\Z_1\times (-2,2/K).$$
Here we quietly assume that $K\gg 1$ so that $2/K-1<-1/K$.
It follows that
\begin{equation*} 
\|\wt N(\nabla (u-w)\1_{J_1})\|_{L^1(\partial \Omega)} \lesssim \|\wt N(\nabla (u-w)\1_{J_1})\|_{L^p(\partial \Omega)} \leq C \iint_{2J_{1}\cap \om} |\nabla(u-w)| dXdt.
\end{equation*}
Here the implied constant in the first estimate is of size $O(1)$ as we are working on a cube of size $\sim 1$.
Using the bound \eqref{Est-w3} we have obtained for $w$ we can now simplify the estimate further to obtain
\begin{equation} \label{u-est1}
\|\wt N(\nabla u\1_{J_1})\|_{L^1(\partial \Omega)} \leq C_K+ C\iint_{(2\Z_1\times(-1/K,2/K))\cap \om} |\nabla u| dXdt,
\end{equation}
where we use for the last integral the fact that when $t<-1/K$ then $\nabla u\equiv 0$.
We can repeat the calculation for other sets $J_i= \Z_i\times (2/K-1,2/K)$, $i=2,3,\dots, N$, this time applying Theorem \ref{thm.NtoLoc} to the pair of cubes $\Z_i\times (2/K-1,2/K)$ and $2\Z_i\times (-2,2/K)$. We distinguish two cases, according to whether the support of $g$ meets the larger of the two cubes.

If $\mbox{supp }g\cap [2\Z_i\times (-1/K,1/K)]=\emptyset$, then $u$ itself has vanishing Neumann data on $[2\Z_i\times (-2,2/K)]\cap\pom$, so we can take $w=0$ and get an estimate
\begin{equation} \label{u-est1b}
\|\wt N(\nabla u\1_{J_i})\|_{L^1(\partial \Omega)} \leq  C\iint_{(2\Z_i\times(-1/K,2/K))\cap \om} |\nabla u| dXdt,
\end{equation}
where the constant $C$ does not depend on $K$.

If, on the other hand, $\mbox{supp }g\cap [2\Z_i\times (-1/K,1/K)]\ne\emptyset$, we may not take $w=0$, as the Neumann data of $u$ need not vanish on the larger cube. Instead we argue as follows. Since $\mbox{supp }g\subset Q_r(0,0)$ with $r<c_n$, the spatial projection of $\mbox{supp }g$ has diameter less than $1/2$, and as it meets $\Z_i$, it is contained in the set $\{|x'|\le 3\}$ in the coordinate system of $\Z_i$ --- well inside the window $\{|x'|\le 8\}$ on which $8\Z_i\cap\partial\mathcal O$ is the graph of the Lipschitz function $\phi_i$. Therefore the construction of the auxiliary unbounded domain and of the auxiliary solution given in the previous two subsections can be repeated with $\Z_i$ in place of $\Z_1$, with the radii $4$, $6$, $8$ used there adjusted by a fixed amount to account for the fact that the support of $g$ is no longer centered at the origin of the coordinate system (the window $\{|x'|\le 8\}$ leaves enough room for this). This produces a solution $w_i$ of an auxiliary PDE on a graph domain $\tilde\Omega_i$ with Neumann datum $g$ such that the PDEs for $u$ and $w_i$ coincide on a neighbourhood of $[2\Z_i\cap\mathcal O]\times\R$, the boundaries of $\Omega$ and $\tilde\Omega_i$ coincide on a neighbourhood of $2\Z_i\cap\partial\mathcal O$, and, by the argument leading to \eqref{Est-w3},
\begin{equation}\label{Est-wi}
\|\wt N(\nabla w_i) \|_{L^1((4\Z_i\times (-2,2))\cap \partial\tilde\Omega_i)}\le C_K.
\end{equation}
Since $\partial_\nu^{A}(u-w_i)=g-g=0$ on $[2\Z_i\times (-2,2/K)]\cap\pom$ (note that the Neumann data of the difference vanish even across the support of $g$, as $u$ and $w_i$ have the same datum there), Theorem \ref{thm.NtoLoc} applies to $u-w_i$ and, exactly as in the derivation of \eqref{u-est1}, we obtain
\begin{equation} \label{u-est1c}
\|\wt N(\nabla u\1_{\Z_i\times (2/K-1,2/K)})\|_{L^1(\partial \Omega)} \leq C_K+ C\iint_{(2\Z_i\times(-1/K,2/K))\cap \om} |\nabla u| dXdt.
\end{equation}
Observe also that the same value of $K$ can be used for every $i$, since all the implied constants in the construction of $w_i$ depend only on $n$, $\lambda$, $\Lambda$, $\mathcal O$ and the constants in the \Np$^{L}$ and \Dq$^{L^*}$ estimates, uniformly in $i$. Hence after summing over all $i=1,2,\dots, N$, using the fact that $\nabla u\equiv 0$ for $t<-1/K$ and observing that the contribution of the interior region $(\mathcal O\setminus\bigcup_i\Z_i)\times(-1/K,2/K)$ (which the sets $J_i$ do not cover) to $\wt N(\nabla u\1_{t<2/K})$ is also controlled by $C\iint_{\mathcal O\times(-1/K,2/K)}|\nabla u|\, dXdt$ (using the interior reverse H\"older inequality for the gradient of a solution), we get that
\begin{equation} \label{u-est2}
\|\wt N(\nabla u\1_{t<2/K})\|_{L^1(\partial \Omega)} \leq NC_K+ C\iint_{\mathcal O\times(-1/K,2/K)} |\nabla u| dXdt.
\end{equation}

Here only the constant $C_K$ depends on $K$. Next we introduce 
$$\mathcal O_\varepsilon=\{X\in \mathcal O:\, \delta(X)>\varepsilon\},$$
and claim that 
$$\iint_{(\mathcal O\setminus\mathcal O_\varepsilon)\times(-1/K,2/K)} |\nabla u| dXdt\lesssim \varepsilon
\|\wt N(\nabla u\1_{t<2/K})\|_{L^1(\partial \Omega)}.$$
It might be useful to compare this estimate with \eqref{zex} and after choosing $\varepsilon>0$ small
(it can be taken the same as before) we obtain
\begin{equation} \label{u-est3}
\|\wt N(\nabla u\1_{t<2/K})\|_{L^1(\partial \Omega)} \leq 2NC_K+ 2C\iint_{\mathcal O_\varepsilon\times(-1/K,2/K)} |\nabla u| dXdt.
\end{equation}
This is analogous to \eqref{eq5.8}. Unsurprisingly, we can also repeat the calculation done in \eqref{rex}--\eqref{rex2} to obtain
\begin{multline*}
\iint_{\mathcal O_\varepsilon\times(-1/K,2/K)} |\nabla u| dXdt\le C\left(
\iint_{\mathcal O_\varepsilon\times(-1/K,2/K)} |\nabla u|^2 dXdt\right)^{1/2}\\
\le C\varepsilon^{-2}K^{-1}\iint_{\mathcal O_{\varepsilon/5}\times(-1/K,2/K)} |\nabla u| dXdt.
\end{multline*}
As the last term of the estimate above can be estimated by the first term of \eqref{u-est3} it follows that
\begin{equation} \label{u-est3-pre}
\|\wt N(\nabla u\1_{t<2/K})\|_{L^1(\partial \Omega)} \leq 2NC_K+ C_\varepsilon K^{-1}\|\wt N(\nabla u\1_{t<2/K})\|_{L^1(\partial \Omega)} .
\end{equation}
This is analogous to \eqref{eq5.11a}. As $K\gg1$ is large enough such that $C_\varepsilon K^{-1}\le 1/2$ (the same $K$ as chosen before will work here since the implied constants for $w$ and $u$ are comparable), we get that
\begin{equation} \label{u-est3-final}
\|\wt N(\nabla u\1_{t<2/K})\|_{L^1(\partial \Omega)} \leq 4NC_K.
\end{equation}
We are nearly done. Observe that we have that $\partial_\nu^Au=0$ on $\partial\mathcal O\times (1/K,\infty)$.
It is therefore natural to use Lemma \ref{lemma:Exponential Decay} to estimate
$\|\wt N(\nabla u\1_{t\ge 2/K})\|_{L^1(\partial \Omega)} $. By \eqref{u-est3-final} we do have a good control 
of $\nabla u$ on the set $\mathcal O\times (1/K,2/K)$ where the Neumann data vanish. The difference is that 
this is a set of size $1/K$ in time (instead of size $1$). However, it is not hard to see that by combining Caccioppoli's inequality, boundary H\"older regularity of solutions and the maximum principle we have that
$$\|\wt N(\nabla u\1_{1/K<t<1/K+1})\|_{L^1(\partial \Omega)}\lesssim_K\|\wt N(\nabla u\1_{1/K<t<2/K})\|_{L^1(\partial \Omega)},$$
and therefore after using Lemma \ref{lemma:Exponential Decay}  we finally obtain the desired estimate
\begin{equation} \label{u-est4}
\|\wt N(\nabla u)\|_{L^1(\partial \Omega)} \leq \tilde C_K.
\end{equation}
This concludes our argument for $r>0$ small. Hence Theorem \ref{MT} holds.\qed

\bibliographystyle{alpha}
%\bibliography{reference}

\begin{thebibliography}{30}

\bibitem{Aro68} D. Aronson, {\em Non-negative solutions of linear parabolic equations}, Ann. Scuola Norm. Sup. Pisa Cl. Sci. (3) 22 (1968), 607--694.

\bibitem{AEN} P. Auscher, M.  Egert, K. Nystr\"om, {\em
$L^2$ well-posedness of boundary value problems for parabolic systems with measurable coefficients}, 
 J. Eur. Math. Soc. 22 (2020), no. 9, 2943--3058.

\bibitem{B} R. Brown, {\em The initial-Neumann problem for the heat equation in Lipschitz cylinders},
Trans. Amer. Math. Soc. 320 (1990), no. 1, 1--52.

\bibitem{Din23} M. Dindo\v{s}, {\em On the regularity problem for parabolic operators and the role of half time derivative}, J. Geom. Anal. 35 (2025), no. 5, Paper No. 154, 25 pp.

\bibitem{DiS} M. Dindo\v{s}, E. Nystr\"om, {\em A relation between the Dirichlet  and the Regularity problem for Parabolic equations}, J. Differential Equations 455 (2026), Paper No. 113962, 58 pp.

\bibitem{DLP2} M. Dindo\v{s}, L. Li, J. Pipher, {\em The $L^p$ Neumann problem for parabolic operators with coefficients satisfying small Carleson condition}, preprint arXiv:2606.09614 (2026).

\bibitem{DLP3} M. Dindo\v{s}, L. Li, J. Pipher, {\em Localization and interpolation of parabolic $L^p$ Neumann problems}, arXiv:2601.12429 (2026). To appear in Recent Advances in Harmonic Analysis and Partial
Differential Equations, in the book series {\it Applied and Numerical Harmonic Analysis},  Springer.

\bibitem{FL} J. Feneuil, L. Li, {\em The $L^p$ Poisson-Neumann problem and its relation to the Neumann problem}, preprint arXiv:2406.16735 (2024).

\bibitem{Ke94} C. E. Kenig, {\em Harmonic analysis techniques for second order elliptic boundary value problems}, CBMS Regional Conference Series in Mathematics, vol. 83, American Mathematical Society, Providence, RI, 1994.

\bibitem{KP93} C. Kenig, J. Pipher, {\em The Neumann problem for elliptic equations with nonsmooth coefficients}. 
Invent. Math. 113 (1993), no. 3, 447–509.

\bibitem{Lieberman} G. M. Lieberman, {\em Second Order Parabolic Differential Equations}, World Scientific Publishing Co., Inc., River Edge, NJ, 1996.

\bibitem{Ny2016} K. Nystr\"om,
{\em Square function estimates and the Kato problem for second order parabolic operators in $\mathbb R^{n+1}$}, Adv. in Math. 293 (2016), 1-36.

\end{thebibliography}

\end{document}